\documentclass[12pt]{amsart}    

\usepackage[top=1.5in, bottom=1.5in, left=1.25in, right=1.25in]{geometry}

\title{Newton Polytopes of Cluster Variables of Type $A_n$}
\author{Adam Kalman}

\address{Department of Mathematics, University of California, Berkeley, 
Evans Hall Room 743, Berkeley, CA 94720}
\email{akalman@math.berkeley.edu}


\usepackage{bbm}
\usepackage{epstopdf, mathtools, dsfont}
\usepackage{caption, subcaption}
\usepackage{amsmath, amsthm, amssymb, amsbsy}
\usepackage{amsfonts, latexsym, stmaryrd, amscd, xy}
\usepackage[mathscr]{eucal}
\usepackage{epsfig}

\newtheorem{theorem}{Theorem}[section]

\newtheorem{proposition}[theorem]{Proposition}
\newtheorem{lemma}[theorem]{Lemma}

\newtheorem{example}[theorem]{Example}
\newtheorem{corollary}[theorem]{Corollary}

\newtheorem{remark}[theorem]{Remark}

\newtheorem{definition}[theorem]{Definition}

\newtheorem{notation}[theorem]{Notation}

\newtheorem*{thm1}{Theorem \ref{latticesIsom}}
\newtheorem*{thm2}{Theorem \ref{NTgFromTg}}

\newcommand{\cc}{\gamma}

\newcommand{\R}{\mathbb R}

\newcommand{\indicator}{\mathbbm{1}}

\newcommand{\GTg}{G_{T,\cc}}
\newcommand{\PG}{P(G_{T,\cc})}
\newcommand{\NTg}{N(T,\gamma)}
\newcommand{\Dg}{D(\gamma)}
\newcommand{\Sbar}{\overline{S}}

\newcommand{\thmrefer}[1]{\renewcommand\thetheorem
  {\protect\ref{#1}}\addtocounter{theorem}{-1}}

\xyoption{all}
\CompileMatrices

\begin{document}


\begin{abstract}
We study Newton polytopes of cluster variables in type $A_n$ cluster algebras, whose cluster and coefficient variables are indexed by the diagonals and boundary segments of a polygon. Our main results include an explicit description of the affine hull and facets of the Newton polytope of the Laurent expansion of any cluster variable, with respect to any cluster. In particular, we show that every Laurent monomial in a Laurent expansion of a type $A$ cluster variable corresponds to a vertex of the Newton polytope. We also describe the face lattice of each Newton polytope via an isomorphism with the lattice of elementary subgraphs of the associated snake graph. 
\end{abstract}

\maketitle

\setcounter{tocdepth}{1}
\tableofcontents

\section{Introduction} \label{Intro}

Cluster algebras, introduced by Fomin and Zelevinsky in the early 2000's \cite{FZ1}, are a class of commutative rings equipped with a distinguished set of generators (\emph{cluster variables}) that are grouped into sets of constant cardinality $n$ (the \emph{clusters}). A cluster algebra may be defined from an initial cluster $(x_1,. . .,x_m)$ and a quiver, which contains combinatorial data for the process of \emph{mutation}, in which new clusters and quivers are created recursively from old ones. There may also be coefficients involved in the construction. The \emph{cluster algebra} is the algebra generated by all cluster variables, after mutation is repeated ad infinitum. 

The original motivation for cluster algebras was to create a combinatorial framework for studying total positivity and dual canonical bases in semisimple Lie groups. Since then, cluster structures have been found and studied in various areas of mathematics. 

Perhaps the most fundamental example of a cluster algebra is the cluster algebra associated with triangulations of a polygon. Cluster algebras of finite type (i.e. those with finitely many cluster variables) are classified by Dynkin diagrams, and the cluster algebras coming from triangulations of a polygon are precisely those of type $A$. In this model, diagonals correspond to cluster variables, triangulations (i.e. maximal collections of non-intersecting diagonals) correspond to clusters, boundary segments correspond to coefficient variables, and mutation corresponds to a local move called a \emph{flip} of the triangulation, in which one diagonal is replaced with another one.

A consequence of the definition of cluster algebra is that every cluster variable is a rational function in the initial cluster variables, but more strongly, the remarkable \emph{Laurent Phenomenon} \cite{FZ1} states that every cluster variable is in fact a Laurent polynomial in those variables. 

In the last ten years, much work has been done on Laurent expansion formulas for cluster algebras. Carroll and Price (in unpublished results \cite{CP}) were the first to discover formulas for Laurent expansions of cluster variables in the case of a triangulated polygon, writing one formula in terms of paths and another in terms of perfect matchings of so-called \emph {snake graphs} \cite{Pr}. Their formula was subsequently rediscovered and generalized in a series of works \cite{S1}, \cite{ST}, \cite{MS}, \cite{MSWp}, with \cite{MSWp} providing Laurent expansions of cluster variables associated to cluster algebras from arbitrary surfaces.

In this paper, we study the Newton polytope of the Laurent expansion of a cluster variable in a type $A$ cluster algebra with respect to an arbitrary cluster. The study of Newton polytopes of Laurent expansions of cluster variables was initiated by Sherman and Zelevinsky in their study of rank 2 cluster algebras, in which it was shown that the Newton polygon of any cluster variable in a rank 2 cluster algebra of finite or affine type is a triangle \cite{SZ}. We will extend these results in type $A$ by considering cluster algebras of arbitrary rank. Another motivation for the study of these Newton polytopes is that understanding Newton polytopes of cluster variables has been useful for understanding bases of cluster algebras \cite{DT, SZ}.
 
Our main results in this paper are explicit descriptions of the affine hull and facets of the Newton polytope of a Laurent expansion of any cluster variable of type $A_n$, as well as a description of the face lattice of such a polytope via an isomorphism with the lattice of elementary subgraphs of the associated snake graph. Our affine hull and facet description can be read off the triangulation directly. We also show that every Laurent monomial in a Laurent expansion of a type $A$ cluster variable corresponds to a vertex of the Newton polytope.

The structure of this paper is as follows. Section \ref{ClusterAlgs} contains a summary of the formula that gives cluster expansions using perfect matchings. Section \ref{Main} lists both major results of this paper, after establishing definitions and notation necessary to state those results. Section \ref{Proofs} includes proofs of the main results, establishing more peripheral results along the way. In Section \ref{Other}, we discuss progress toward more general results.

\textsc{Acknowledgements:} The author would like to thank Lauren Williams, Gregg Musiker, Dylan Thurston, and Christian Hilaire for their helpful ideas. I am particularly grateful to Lauren Williams for many enlightening conversations, ideas, and suggestions. 

\section{Cluster Expansions from Matchings} \label{ClusterAlgs}

The cluster algebra we are considering in this paper is constructed from a triangulation $T$ of an $(n+3)$-gon as follows. Let $\tau_1,\tau_2,\ldots,\tau_n$ be the $n$ diagonals of
$T$, and let $\tau_{n+1},\tau_{n+2},\ldots,\tau_{2n+3}$ be the $n+3$ boundary segments.

The quiver $Q_T$ is defined as follows: place a frozen vertex at the midpoint of each boundary segment of the polygon, and place a mutable vertex at the midpoint of each diagonal. These midpoint vertices form the vertices of $Q_T$. Label these vertices according to the labeling of the polygon. To form the arrows of $Q_T$, go to each triangle of $T$ and inscribe a new triangle connecting the midpoint vertices, orienting the arrows clockwise within this new triangle. For example, here is a triangulation $T$ of a hexagon, along with the corresponding quiver $Q_T$. The diagonals and boundary segments of $T$ are shown as thin solid lines. Mutable vertices of the quiver are indicated by filled-in circles, frozen vertices are indicated by unfilled circles, and the arrows of the quiver are dashed lines.

\begin{figure}[h!]
\centering
\includegraphics[scale=0.4]{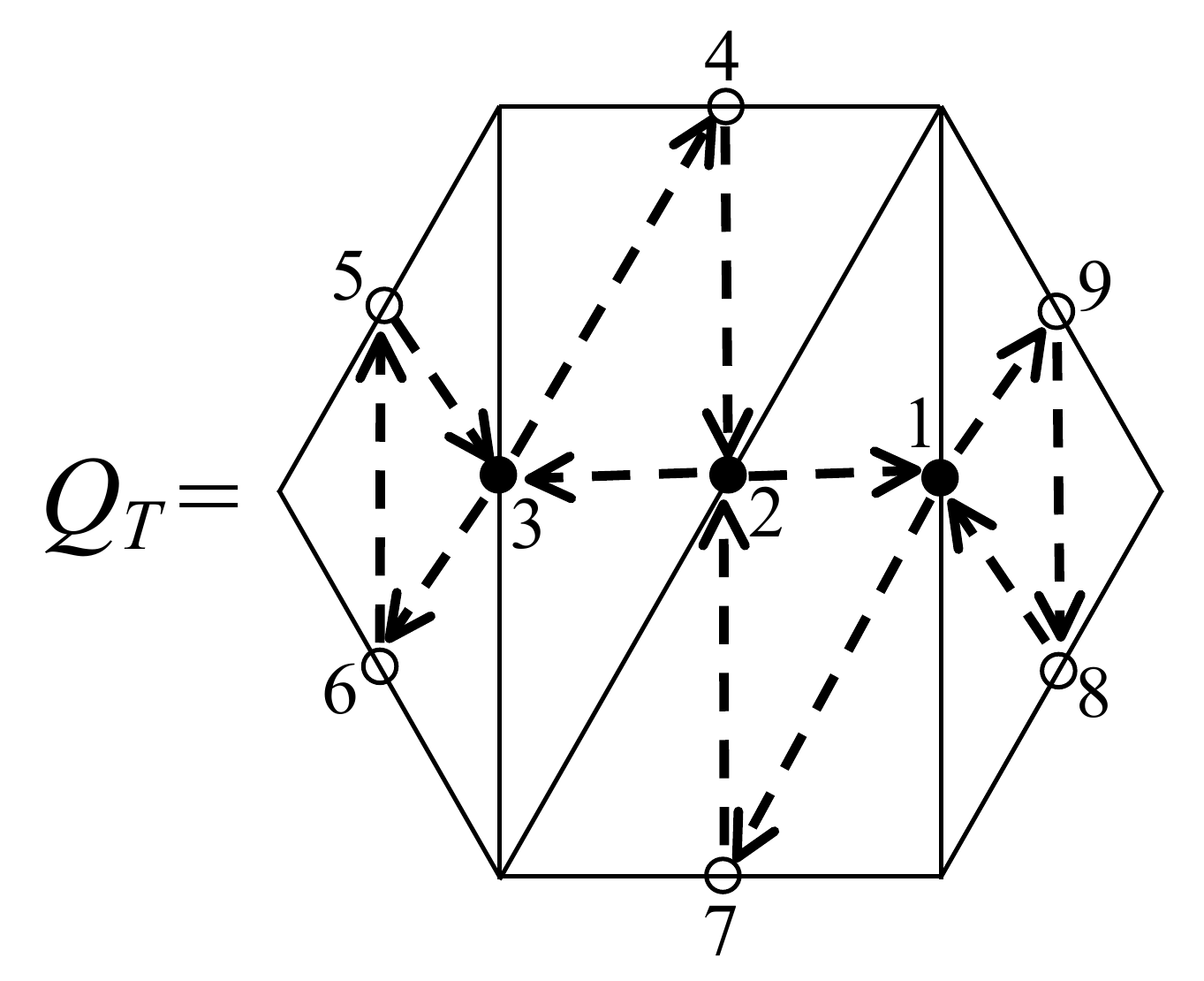}
\label{fig:quiver}
\end{figure}

Let $\mathcal{A}(Q_T)$ be the cluster algebra with initial cluster variables $(x_1,...,x_n)$, coefficient variables $(x_{n+1},...,x_{2n+3})$, and initial quiver $Q_T$. Each cluster variable in $\mathcal{A}(Q_T)$ corresponds to a diagonal. Let $x_\cc$ be the cluster variable corresponding to the diagonal $\cc$. 

The \emph{cluster expansion of $x_\cc$ with respect to $T$}, or the \emph{$T$-expansion of $x_\cc$},  means the Laurent polynomial (equal to $x_\cc$) in the variables which each correspond to a diagonal or boundary segment of $T$. The formula for the $T$-expansion of $x_\cc$ in \cite{MS} for the cluster variables is given in terms of perfect matchings of a graph $G_{T,\cc}$ that is constructed using recursive gluing of \emph{tiles}. We now recount the construction of this graph $G_{T,\cc}$, as described in \cite{MS} and \cite{MSW}.

Let $\cc$ be a diagonal which is not in $T$. Choose an orientation on $\cc$, and let the points of intersection of $\cc$ and $T$, in order, be $p_0, p_1, \ldots, p_{d+1}$. Let $\tau_{i_1}, \tau_{i_2}, \dots, \tau_{i_d}$ be the diagonals of $T$ that are crossed by $\cc$, in order.

For $k$ from $0$ to $d$, let $\cc_k$ denote the segment of the path $\cc$ from $p_k $ to $p_{k+1}$. Note that each $\cc_k$ lies in exactly one triangle in $T$, and for $1\le k\le d-1$, the sides of this triangle are $\tau_{i_k}, \tau_{i_{k+1}}$, and a third edge denoted by $\tau_{[\cc_k]}$.

A \emph{tile} $\Sbar_k$ is a 4-vertex graph consisting of a square along with one of its diagonals. Any diagonal $\tau_k\in T$ is the diagonal of a unique quadrilateral $Q_{\tau_k}$ in $T$ whose sides we will call $\tau_a, \tau_b, \tau_c,\tau_d$. Associate to this quadrilateral a tile $\Sbar_k$ by assigning weights to the diagonal and sides of $\Sbar_k$ in such a way that there is a homeomorphism $Q_{\tau_k} \to \Sbar_k$ which maps the diagonal labeled $\tau_i$ to the edge with weight $x_i$, for $i=a,b,c,d,k$. 

For each tile $\Sbar_{i_1},\Sbar_{i_2},\ldots, \Sbar_{i_d}$, we choose a planar embedding in the following way: For $\Sbar_{i_1}$, the homeomorphism $Q_{\tau_{i_1}} \to \Sbar_{i_1}$ must be orientation-preserving, and the vertex of $\Sbar_{i_1}$ which corresponds to $p_0$ is placed in the southwest corner. Then, for $2 \leq k \leq d$, choose a planar embedding for $\Sbar_{i_k}$ which has the opposite orientation of the previous tile $\Sbar_{i_{k-1}}$, and orient the tile $\Sbar_{i_k}$ so that the diagonal goes from northwest to southeast.

We then create the graph $\overline{G}_{T,\cc}$ by gluing together tiles $\Sbar_{i_1},\Sbar_{i_2},\ldots, \Sbar_{i_d}$, in order, attaching $\Sbar_{i_{k+1}}$ to $\Sbar_{i_k}$ along the edge on each tile that is labeled $x_{[\cc_k]}$. Note that the edge weighted  $x_{[\cc_k]}$ is either the northern or the eastern edge of the tile $\Sbar_{i_k}$, and hence $\overline{G}_{T,\cc}$ is constructed from the bottom left (the first tile) to the upper right (the last tile).

\begin{definition}
The \emph{snake graph} $\GTg$ is the graph obtained from $\overline{G}_{T,\cc}$ after the diagonal is removed from each tile. 
\end{definition}

See Figure \ref{fig:example} for an example of a triangulation $T$ (along with distinguished diagonal $\cc$) and the corresponding snake graph. 

\begin{figure}[h!]
        \centering
        \begin{subfigure}[b]{0.46\textwidth}
                \centering
                \includegraphics[width=\textwidth]{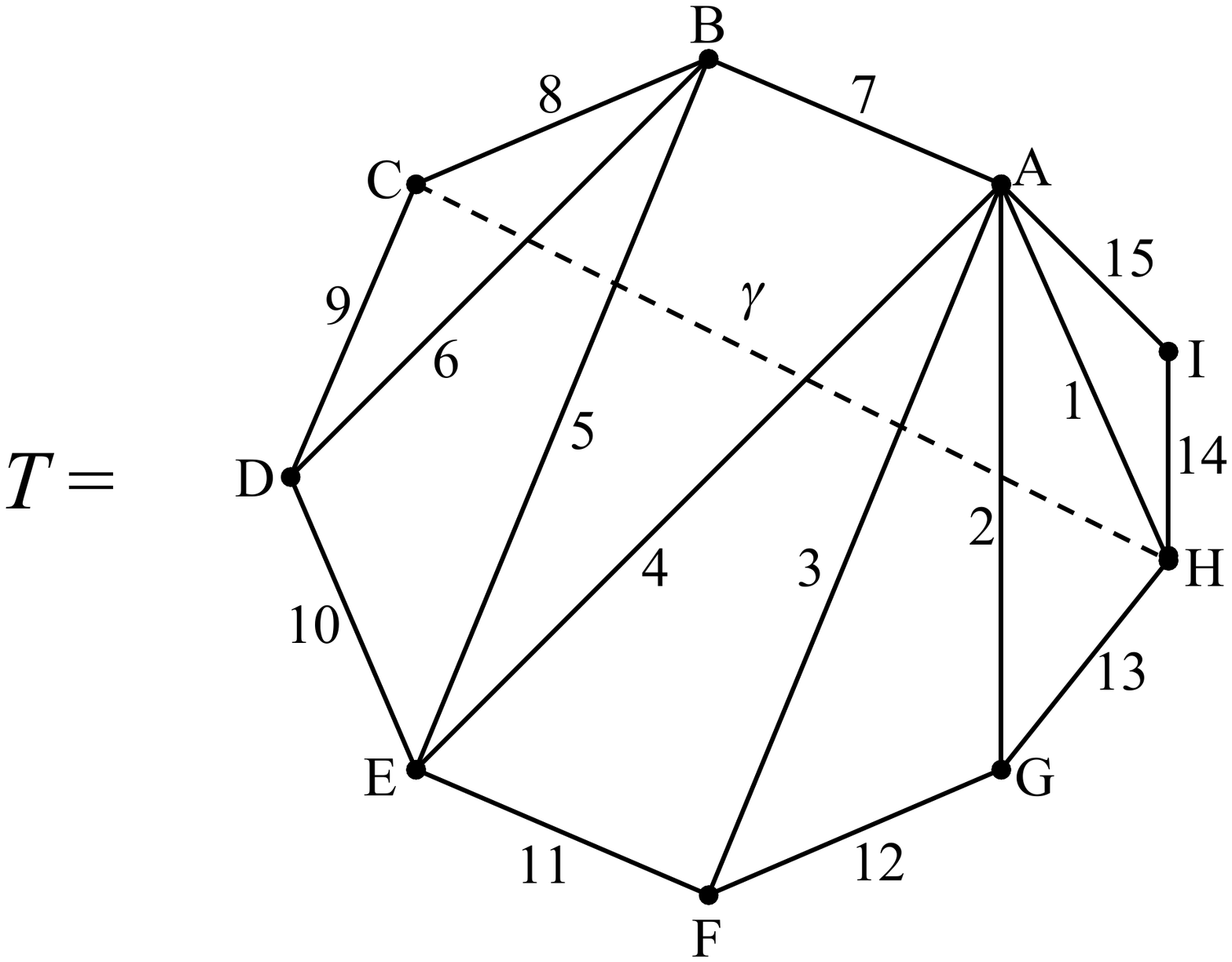}
                \label{fig:T-alpha}
        \end{subfigure}
        \qquad
        \begin{subfigure}[b]{0.47\textwidth}
                \centering
                \includegraphics[width=\textwidth]{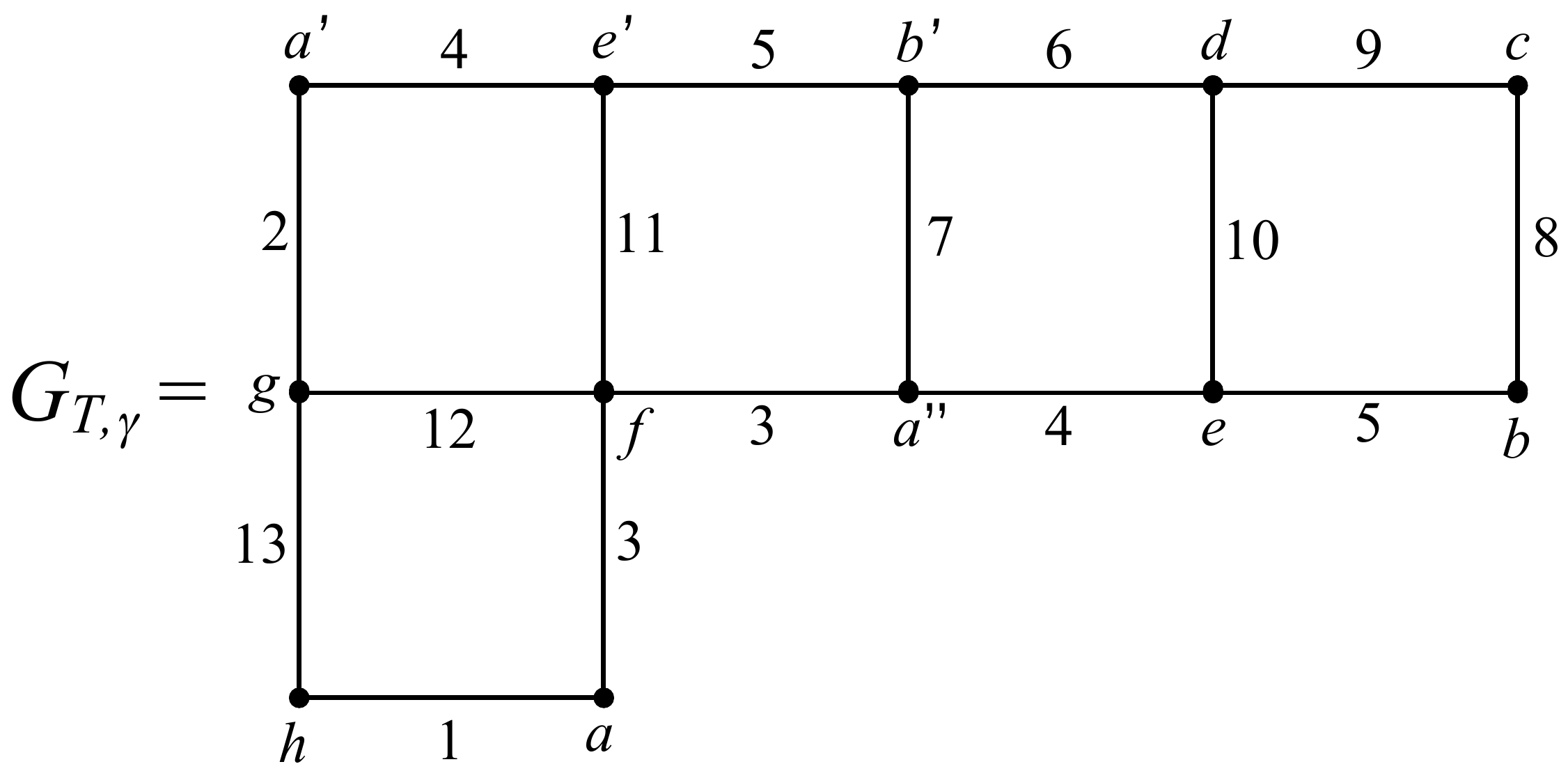}
                \label{fig:G-alpha}
        \end{subfigure}
        \caption{An Example}\label{fig:example}
\end{figure}

\begin{definition}
A \emph{perfect matching} of a graph is a subset of the edges so that each vertex is covered exactly once.  The \emph{weight} $w(M)$ of a perfect matching $M$ is the product of the weights of all edges in $M$. 
\end{definition}

With this setup, Laurent expansions of cluster variables can be expressed in terms of perfect matchings as follows (see \cite{MS}).

\begin{proposition} \label{MSmatching} 
With the above notation,
\[x_\cc = \sum_M \frac{w(M)}{x_{i_1}x_{i_2}\ldots x_{i_d}}, \]
where the sum is over all perfect matchings $M$ of $G_{T,\cc}$.
\end{proposition}  

\section{Main Results} \label{Main}

Before we can state our main results, we need a few more definitions and some new notation.

\begin{definition}
The \emph{Newton polytope} of a Laurent polynomial is the convex hull of all the exponent vectors of the monomials, i.e. the convex hull of all points $(c_1,c_2,. . .)$ such that the monomial $x_1^{c_1} x_2^{c_2} . . .$ appears with a nonzero coefficient in the Laurent polynomial.
\end{definition}

For ease of notation, we may sometimes say a diagonal or boundary segment of the polygon is labeled $k$ rather than $\tau_k$. 

\begin{notation} {\rm
\begin{itemize}
\item Let $\Dg$ denote the set of diagonals of the triangulation that $\cc$ crosses, i.e. $ \{\tau_{i_1},\tau_{i_2},\ldots,\tau_{i_d} \}$. 
\item Let $T'$ be the subset of $T$ that includes all vertices incident to a diagonal in $\cc \cup \Dg$, and all diagonals and boundary segments connecting these vertices to each other. 
\item For any point $w \in T$, let $diagonals(w) :=  \{e \in (\cc \cup \Dg): e \ni w\}$, the set of diagonals in $\cc \cup \Dg$ incident to $w$.
\item Let the set of distinct labels of edges incident to a vertex $v \in \GTg$ be $E_v$. If $V$ is a collection of vertices, let $E_V := \bigcup_{w \in V} E_w$
\item Let $\NTg$ be the Newton polytope (in $\mathbb{R}^{2n+3}$) of the $T$-expansion of the cluster variable $x_\cc$. 
\item Let $\PG$ be the polytope in $\mathbb{R}^{2n+3}$ that is the convex hull of the characteristic vectors of all perfect matchings of $\GTg$. 
\end{itemize} }
\end{notation}

\begin{remark} \label{PGaNewtPolyOfNumerator}
By Proposition \ref{MSmatching}, the two polytopes $\NTg$ and $\PG$ are isomorphic, differing only by a translation by the vector $\indicator_{\Dg}$ (i.e. the vector whose $i^{th}$ coordinate is 1 if $i \in \Dg$, 0 otherwise). So $\PG$ can be thought of as the ``Newton polytope of the numerator" of the cluster variable corresponding to $\cc$.
\end{remark}

\begin{definition}
Define an equivalence relation $\sim$ on the set of vertices of $\GTg$ as follows: Vertices of $\GTg$ are equivalent if they correspond to the same marked point on the original polygon $T'$, based on how quadrilaterals from the polygon become tiles in $\GTg$. Let the equivalence class of a vertex $v$ be $[v]$. 
\end{definition}

The location of equivalent vertices follows this specific pattern: $v \sim v'$ if one can start at $v$ and reach $v'$ by a sequence of northwest-southeast knight's moves (i.e. we are allowed to make the ``knight's move" in only 4 directions (not 8): left 1 and up 2, left 2 and up 1, right 1 and down 2, or right 2 and down 1). This can be seen by examining the construction of $\overline{G}_{T,\cc}$. Specifically, we see that two triangles incident to the same tile edge must have the same edge labels, and in that way their vertices naturally correspond. (In \cite{MS}, this phenomenon is used to define a ``folding map.") Also note that this equivalence relation on vertices induces an equivalence relation on edges that corresponds to the non-uniqueness of edge labels in the following way. Suppose vertices $v$ and $w$ are adjacent, with an edge labeled $e$ joining them. If $v \sim v'$ and $w \sim w'$, then vertices $v'$ and $w'$ are adjacent, and the edge joining them is also labeled $e$. Conversely, suppose two edges of $\GTg$ (call them $\{v, w\}$ and $\{v' ,w'\}$) have the same label $e$. Then $v \sim v'$ and $w \sim w'$.

\begin{definition}
A tile $\overline{S}$ in $\GTg$ will be called a \emph{corner} if it is incident to two other tiles, one of which is left or right of $\overline{S}$, and one of which is above or below $\overline{S}$. 
\end{definition}
 
\begin{definition} 
A diagonal $e$ in $\Dg$ will be called \emph{balanced} if a pair of opposite sides of $Q_{\tau_e}$ consists of boundary segments of $T'$, and will be called \emph{imbalanced} otherwise. 
\end{definition} 

\begin{definition}
A subgraph $H$ of a bipartite graph $G$ will be called an \emph{elementary subgraph} if $H$ contains every vertex of $G$, and every edge of $H$ is used in some perfect matching of $H$. Equivalently, $H$ is an elementary subgraph if it is the union of some set of perfect matchings of $G$.
\end{definition}

To illustrate some of this vocabulary, we will use the triangulation and snake graph from Figure \ref{fig:example} as an example. 

The set $\Dg = \{2,3,4,5,6\}$, and $T'$ is the graph below. The vertices of $\GTg$ in Figure \ref{fig:example} are labeled with lowercase letters to correspond in a natural way to the vertices of $T'$ (or $T$), labeled in uppercase.

\begin{figure}[h!]
\centering
\includegraphics[scale=0.4]{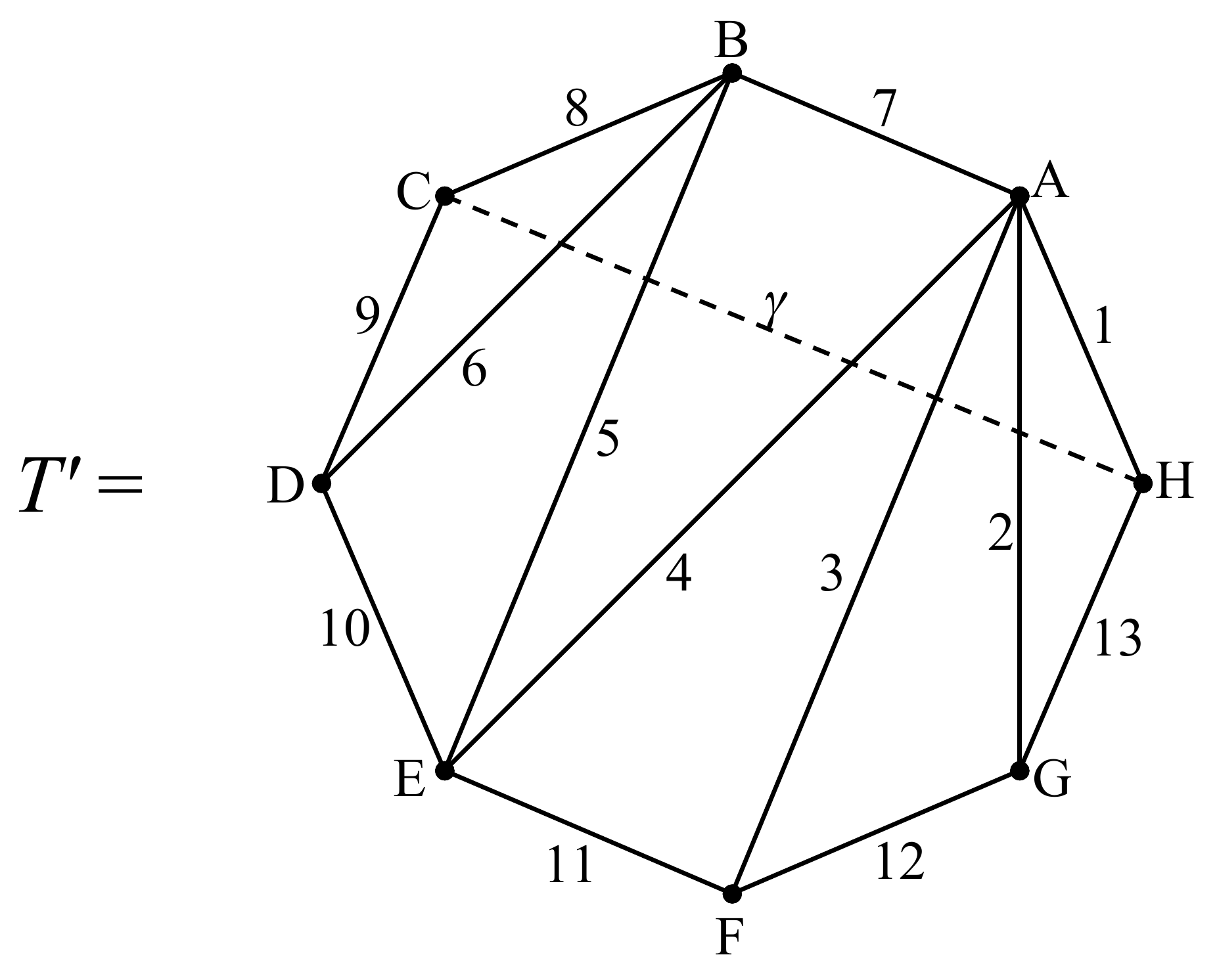}
\label{fig:S-alpha}
\end{figure}

The equivalence class $[a']$ is $\{a,a',a''\}$. Observe the northwest-southeast knight's moves between these vertices of $\GTg$, and notice that this equivalence class corresponds to vertex $A$ of $T'$.

Also, $E_{[a']} = \{1,2,3,4,7\}$, and $diagonals(A) = \{2,3,4\}$.

Note that in $T'$, the diagonal connecting vertices $B$ and $E$ is labeled ``5". Correspondingly, in $\GTg$, the edges $\{b,e\}$ and $\{b',e'\}$ are both labeled ``5".

Moreover, $Q_{\tau_5} = (7,4,10,6)$. The diagonal ``5" is balanced. Note that $\GTg$ has 1 corner - the second tile.

To construct $\PG$, we associate a characteristic vector in $\R^{15}$ to every perfect matching of $\GTg$, and find the convex hull of all these vectors. For example, the matching below gives the vector $(1,1,1,0,2,0,0,0,1,0,0,0,0,0,0)$. 

\begin{figure}[h!]
\centering
\includegraphics[scale=0.3]{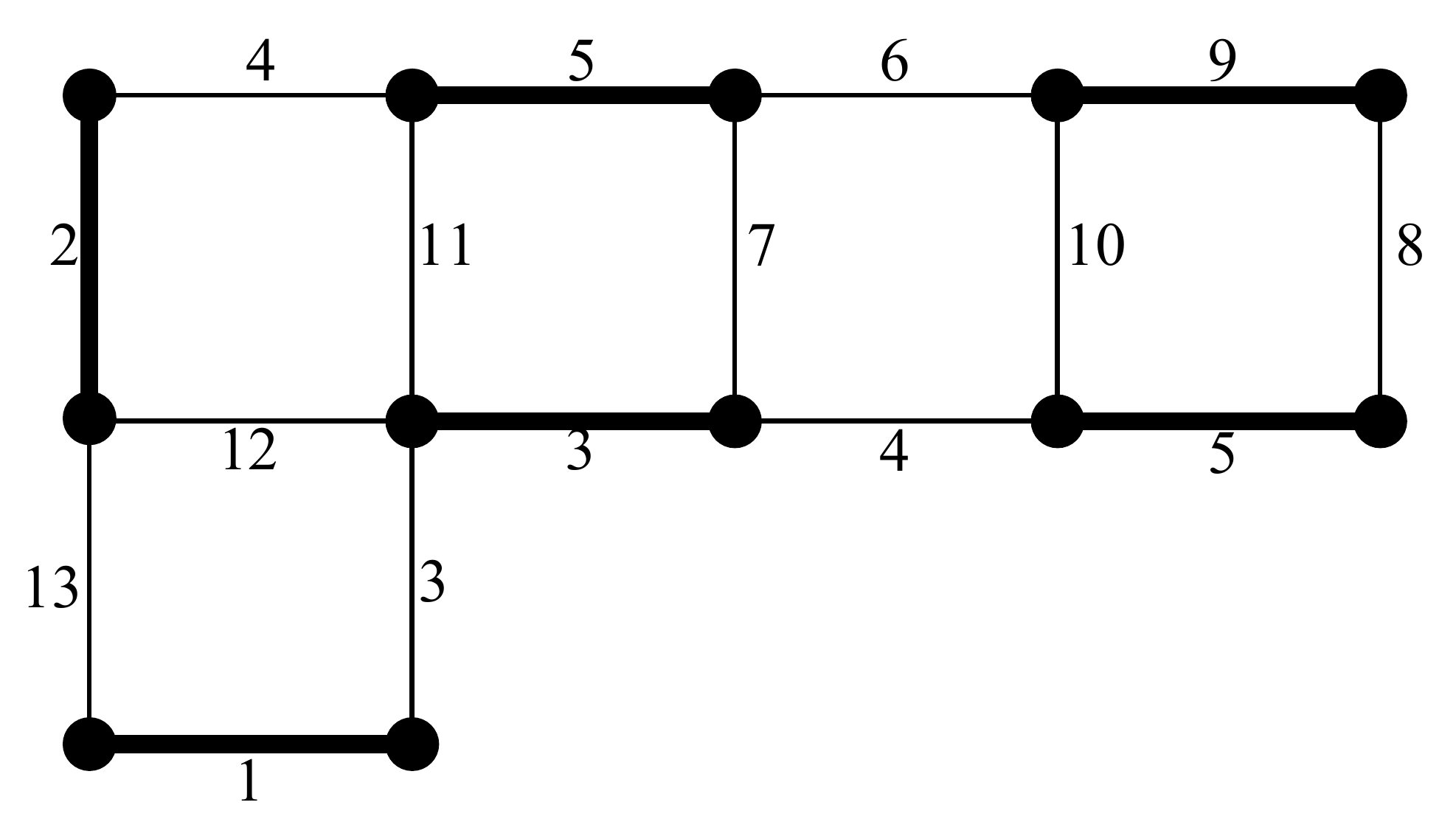}
\label{fig:matching}
\end{figure}

\pagebreak

Our main results in this paper are as follows:

\begin{thm1}
The face lattice of $\NTg$ (and of $\PG$) is isomorphic to the lattice of all elementary subgraphs of $\GTg$, ordered by inclusion.
\end{thm1}

\begin{thm2}
For any diagonal $\cc$, the polytope $\NTg$ can be found directly from $T$ as follows:\\
Affine hull equations:
\begin{align*}
{\rm(i)}& \quad \text{For each edge $e$ of } T \backslash T'  \text{, write } x_e = 0.\\
{\rm(ii)}& \quad \text{For each vertex $w \in T'$, write } \sum_{e \ni w} x_e = 1 \text{ if } w \in \cc, \text{ or write } \sum_{e \ni w} x_e =0 \text{ if } w \notin \cc.
\end{align*}
Facet-defining inequalities:
\begin{align*}
{\rm(iii)}& \quad \text{For every boundary segment $e \in T'$ not incident to } \cc \text{, write }  x_e \geq 0.\\
{\rm(iv)}& \quad \text{For every pair of boundary segments $\{b,c\}$ of $T'$ that are opposite sides}\\
& \quad \text{of $Q_{\tau_a}$, where $a \in \Dg$ is a balanced diagonal, let the other pair of opposite sides }\\
& \quad \text{of $Q_{\tau_a}$ be $\{e,f\}.$ Exactly one of these three cases will hold for each pair $\{e,f\}$:}\\
& \qquad \text{- If } \{e,f\} \subset \{\tau_{i_2},\ldots,\tau_{i_{d-1}}\} \text{, write the inequality } x_a+x_b+x_c \leq 1.\\  
& \qquad \text{- If one of } \{e,f\} \text{ (say $e$) is a boundary segment of } T', \text{ write } x_e \geq 0.\\ 
& \qquad \text{- Otherwise, write } x_e \geq -1, \text{ where $e$ is diagonal } \tau_{i_1} \text { or } \tau_{i_d}.
\end{align*}
\end{thm2}

\section{Other Results and All Proofs} \label{Proofs}

Before we can prove our main results, we need to understand how features of $T'$ correspond to features of $\GTg$. 

\begin{lemma} \label{SaFeatVsGaFeat} {\rm
Let the points of intersection of $\cc$ and $T'$, in order, be $p_0, \ldots, p_{d+1}$, and let $\tau_{i_1}, \dots, \tau_{i_d}$ be the diagonals of $T'$ that are crossed by $\cc$, in order.
When constructing $\GTg$ from $T'$,
\begin{itemize}
\item Each balanced diagonal in $\{\tau_{i_2},\ldots,\tau_{i_{d-1}}\}$ becomes 2 identically labeled exterior edges of $\GTg$ that are parallel and are a northwest-southeast knight\textrm's move apart. There is no corner in $\GTg$ at the tile between the two tiles containing these edges.
\item Each imbalanced diagonal in $\{\tau_{i_2},\ldots,\tau_{i_{d-1}}\}$ becomes 2 identically labeled exterior edges of $\GTg$ that are perpendicular and share a vertex. There is a corner in $\GTg$ at the tile between the two tiles containing these edges.
\item Each boundary segment of the polygon that is not incident to $\cc$ becomes a uniquely labeled interior edge of $\GTg$.
\item The pair of boundary segments of the polygon incident to $p_0$ becomes the bottom and left edges [each uniquely labeled] on the first tile of $\GTg$. The pair of boundary segments of the polygon incident to $p_{d+1}$ becomes the top and right edges [each uniquely labeled] on the last tile of $\GTg$.
\item Assume $T'$ is not a triangulated quadrilateral. The diagonal $\tau_{i_1}$ becomes the uniquely labeled edge $x_{i_1}$ on the left exterior edge or bottom exterior edge (whichever exists) of the second tile of $\GTg $. The diagonal $\tau_{i_d}$ becomes the uniquely labeled edge $x_{i_d}$ on the right exterior edge or top exterior edge (whichever exists) of the penultimate tile of $\GTg$. (If $T'$ is a quadrilateral, the lone diagonal is not present in $\GTg$.)
\item There are a total of $d$ boxes in the snake graph, and $d =  |\Dg|$.
\end{itemize}} 
\end{lemma}

The reader is encouraged to observe how Figure \ref{fig:example} illustrates the above lemma. Specifically, the bullet points refer to, respectively, diagonals 4 and 5; diagonal 3; boundary segments 7, 10, 11, and 12; boundary segment pairs $\{1,13\}$ and $\{8,9\}$; and diagonals 2 and 6. Observe where these labels end up on $\GTg.$

\begin{proof}
The proof follows from the construction of the snake graph as a gluing of tiles isotopic to quadrilaterals in the triangulation. The last statement here is clear by construction of $\GTg$. For the first and second statements, observe that a diagonal $e \in \{\tau_{i_2},\ldots,\tau_{i_{d-1}}\}$ is a side of exactly two non-overlapping quadrilaterals in the triangulation, which become exactly two tiles in $\GTg$. There is exactly 1 other tile between these two tiles in $\GTg$, corresponding to $Q_{\tau_e}$. This results in two identical labels of $e$ on exterior edges. If $e$ is a balanced diagonal, then following the construction process shows that the three tiles are in a straight line, and the two labels of $e$ end up on edges that are a northwest-southeast knight's move apart in $\GTg.$ If $e$ is imbalanced, the three tiles form an L-shape, and the two labels of $e$ end up on adjacent perpendicular edges. 

Similarly, a boundary segment $e$ of $T'$ that is not incident to $\cc$ is a side of exactly two quadrilaterals in the triangulation. These two quadrilaterals are ``consecutive" in that they overlap in a triangle, so they become two tiles in $\GTg$ that are glued together. The side the two quadrilaterals share is $e$, so the tiles are glued along $e$, meaning that $e$ becomes an interior edge of $\GTg$. No other edge of $\GTg$ can be labeled $e$ since no other quadrilateral in the triangulation involves $e$.

Now, let $e$ be an edge of $T'$ that is a boundary segment of the polygon incident to $\cc$. Then $e$ is a side of exactly one quadrilateral in the triangulation, so it becomes a side of exactly one tile in $\GTg$. In the ordering of diagonals according to their intersections with $\cc$, the diagonal $e$ is either first or last, so this must be the first or last tile placed. The label $e$ must be unique (because it only appears in one tile), and must appear on an edge of $\GTg$ that is guaranteed to be exterior regardless of the next tile placed, because interior edges of $\GTg$ belong to two adjacent tiles, which is not the case. All of this forces the four boundary edges meeting this criterion to become the left and bottom of the first tile, and the top and right of the last tile.

Finally, if $T'$ is a triangulated quadrilateral, it is obvious that the lone diagonal is not present in $\GTg$. So assume $T'$ is not a quadrilateral, and let $e$ be the diagonal $\tau_{i_1}$ or $\tau_{i_d}$.  Then $e$ is a side of exactly one quadrilateral in the triangulation, so it becomes a side of exactly one tile in $\GTg$. In the ordering of diagonals according to their intersections with $\cc$, the diagonal $e$ is either second or next-to-last, so this must be the second or next-to-last tile placed. The label $e$ must be unique (because it only appears in one tile), and must appear on an exterior edge of $\GTg$ because interior edges of $\GTg$ belong to two adjacent tiles, which is not the case. The correspondence between vertices of $T'$ and $\GTg$ forces the specific placement described.

\end{proof}

\begin{remark} \label{1or2edgesSameLabel}
Note that in $\GTg$, any number of vertices can be in an equivalence class, but at most two edges have the same label (because any edge in the triangulation is an edge of either 1 or 2 quadrilaterals, and $\cc$ cannot cross the same diagonal more than once).
\end{remark} 

We now state a classical result on bipartite graphs and a related polytope.

\begin{definition}
The \emph{perfect matching polytope} $PM(G)$ of a graph $G$ (with unique edge labels) is the convex hull of the incidence vectors of perfect matchings of $G$. 
\end{definition}

\begin{lemma} \label{LP734}
(\cite{E}, also \cite{LP}, Theorems 7.3.4, 7.6.2): If $G$ is a bipartite graph, then $PM(G)$ is given by the following equations and inequalities:
\begin{align*}
{\rm(i')}\quad &x_e \geq 0 \text{ for each edge } e \text{ of } G \\
{\rm(ii')}\quad &\sum_{e \ni v} x_e = 1 \text{ for each vertex $v$ of } G.
\end{align*}
The dimension of $PM(G)$ is $|E(G)| - |V(G)| + 1.$
\end{lemma}

We will use the above lemmas to prove our first proposition:

\begin{proposition} \label{PGaAffHull}
For any diagonal $\cc$, the affine hull of the polytope $\PG$ can be found from the snake graph $\GTg$ by writing the following equations: 
\begin{align*}
{\rm(i)}\quad &x_e = 0 \text{ for each edge } e \in T \text{ that does not appear in } \GTg \\
{\rm(ii)}\quad &\sum_{e \in E_{[v]}} x_e = |[v]| \text{ for each vertex equivalence class $[v]$ of $\GTg$}
\end{align*}
Using the triangulation $T$ directly, the equivalent equations are
\begin{align*}
{\rm(iii)}\quad &x_e = 0 \text{ for each edge } e \text{ of } T \backslash T' \\
{\rm(iv)}\quad & \sum_{e \ni w} x_e = |diagonals(w)| \text{ for each vertex } w \text{ of } T'
\end{align*}
The equations defining the affine hull here (either {\rm(i)-(ii)} or {\rm(iii)-(iv)}) are linearly independent, and thus are a minimal description of the affine hull. There are $2n+3 - |\Dg|$ equations in this minimal description.
\end{proposition}

\begin{proof}
Proposition \ref{PGaAffHull} will follow immediately from the three statements that follow (Lemma \ref{34indep}, Lemma \ref{12equiv34}, and Proposition \ref{12affHull}).
\end{proof}

\begin{lemma} \label{34indep}
The set of equations {\rm(iii)-(iv)} consists of $2n+3 - |\Dg|$ linearly independent equations.
\end{lemma}

\begin{proof}
Suppose some linear combination of the sums on the left-hand sides of equations (iv) equals zero:
\begin{equation} \label{linComb}
\sum_{w \in T'} \left( c_w \sum_{e \ni w} x_e \right) = 0 
\end{equation}
Here, the outer sum runs over each vertex $w$ of $T'$. Choose a diagonal that forms a triangle with two boundary segments of $T'$, and label it 1. Without loss of generality, we can label $T'$ as in this figure:

\begin{figure}[h!]
\centering
\includegraphics[scale=0.36]{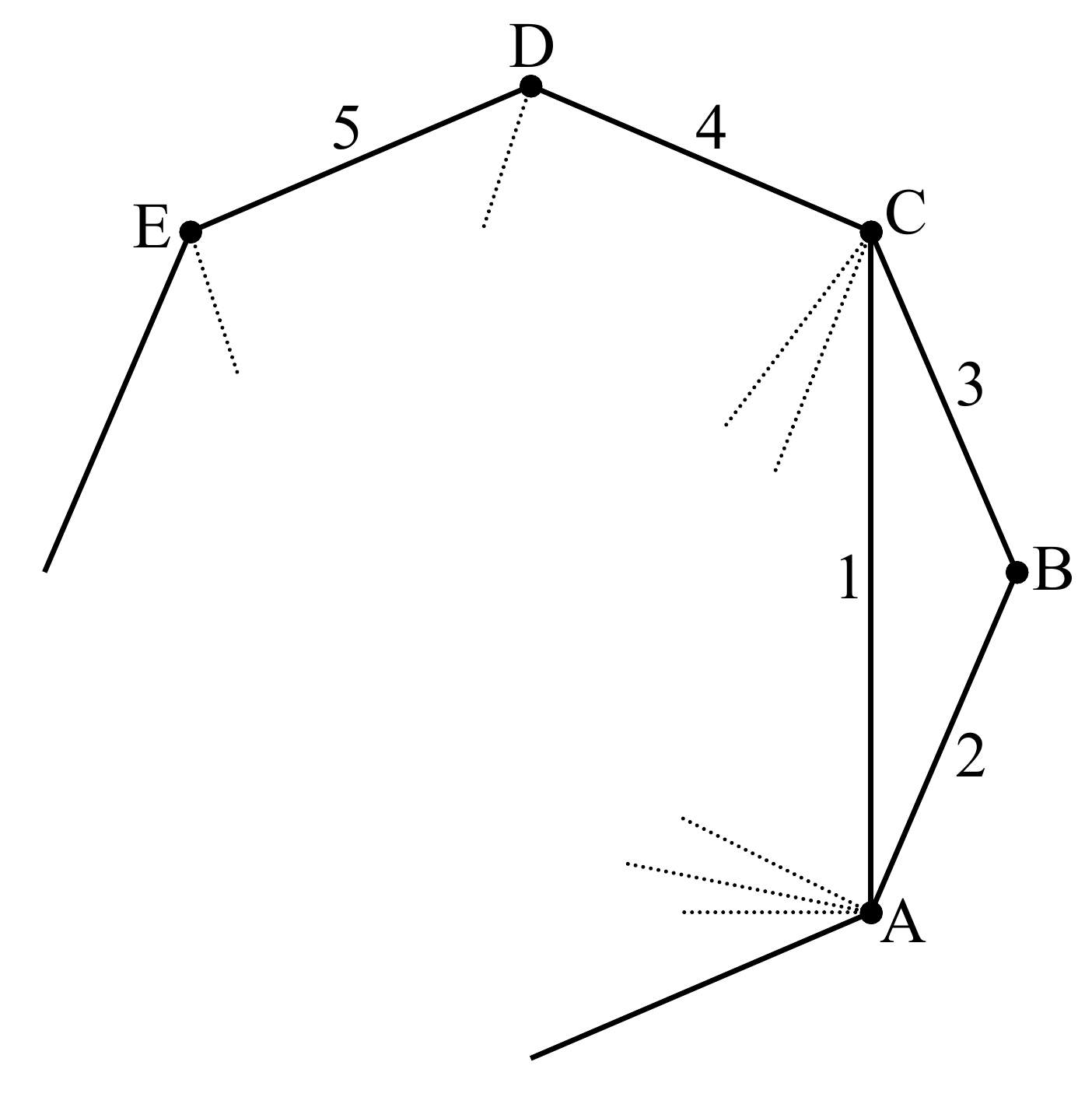}
\label{fig:linindepfig}
\end{figure}

Note that each edge label in $T'$ appears in exactly two sums in Proposition \ref{PGaAffHull}(iv) (the equations corresponding to the two endpoints of that edge). We can thus collect the terms of equation (\ref{linComb}) so that the coefficient of $x_{edge(w_1,w_2)}$ is $c_{w_1} + c_{w_2}$ (see above figure):
\begin{equation} \label{collectedLinComb}
(c_A + c_C)x_1 + (c_A + c_B)x_2 + (c_B + c_C)x_3 + ... = 0. 
\end{equation}

Since all coefficients in equation (\ref{collectedLinComb}) must be zero, we can reason as follows. First, since $c_A + c_C = c_A + c_B = c_B + c_C = 0,$ we can conclude that $c_A = c_B = c_C = 0$. Now, the next term in equation (\ref{collectedLinComb}) is $(c_C + c_D)x_4$, so $(c_C + c_D)$ must be zero. But we know $c_C = 0$, so this forces $c_D = 0$. This knowledge in turn forces $c_E = 0$, and so on until all $c_w$ are zero. Thus the equations (iv) are linearly independent. 

Clearly, the equations (iii) are linearly independent from each other, and also are linearly independent from those in (iv) because they involve different variables. So indeed (iii) and (iv) together form a linearly independent set.

Next we will prove that there are $2n+3 - |\Dg|$ equations in this description (iii)-(iv). Observe that the set of edges of the triangulation $T$ can be partitioned into three subsets:
\begin{align*}
\text{edges of } T &= \text{edges of } T \backslash T' \; \coprod \; \text{ boundary segments of } T' \; \coprod \; \Dg. \\
\text{ Taking } & \text{cardinalities, this equation becomes} \notag \\
2n+3 &= \text{\# of edges of } T \backslash T' +  \text{\# of vertices in } T' + |\Dg|. \text{ This is} \notag \\
2n+3 &= \text{\# of equations in (iii)} +  \text{\# of equations in (iv)} + |\Dg| \text{, as desired.} \notag
\end{align*}
\end{proof}

\begin{lemma} \label{12equiv34}
The set of equations {\rm(i)-(ii)} is equivalent to the set of equations {\rm(iii)-(iv)}.
\end{lemma}

\begin{proof}
We will consider the cases $|\Dg| \geq 2$ and $|\Dg| = 1$ separately.

Suppose first that $|\Dg| \geq 2$. Note that $\GTg$ is constructed from $T'$, so any edge $e$ that is not in $T'$ does not appear in $\GTg$. Conversely, since $|\Dg| \geq 2$, every edge that is in $T'$ is the side of a quadrilateral in $T'$, so it appears in $\GTg$. Thus (i) is equivalent to (iii).

We will show that by the construction of the snake graph as a gluing of tiles isotopic to quadrilaterals in the triangulation, (ii) is equivalent to (iv). Comparing the two statements, we need to show that each vertex equivalence class $[v]$ of $\GTg$ corresponds to a vertex $w$ of $T'$, that the labels of edges incident to $[v]$ in $\GTg$ are the same as the labels of edges incident to the corresponding $w$ in $T'$, and that $|[v]| = |diagonals(w)|$. There are 2 cases to consider.

Case 1: Vertex $w \in T'$ is incident to $\cc$. Then it is not incident to any diagonals in $T'$, so $diagonals(w) = \{ \cc \}$. In this case, $w$ is a vertex of exactly 1 quadrilateral in $T'$, which becomes exactly 1 tile in $\GTg$, and so there is exactly 1 vertex $v$ in the snake graph $\GTg$ corresponding to $w$. The labels of the two edges incident to $w$ in $T'$ clearly become the labels of the two edges incident to $v$ in the tile that becomes embedded in $\GTg$. Since $w$ is not part of any other quadrilateral in $T'$, the corresponding vertex $v \in \GTg$ only appears in this tile in the snake graph, and the next tile is glued onto it along an edge that does not include $v$. Thus no more edges can become incident to $v \in \GTg$ when this next tile is glued on. So the labels of edges incident to $v \in \GTg$ are precisely the same as the labels of edges incident to the corresponding 
$w \in T'$, as desired. Also, in this case $|[v]| = 1 = |\{ \cc \}| = |diagonals(w)|$.

Case 2: Vertex $w \in T'$ is not incident to $\cc$. Then, since $|\Dg| \geq 2, w$ is a vertex of at least 2 quadrilaterals in $T'$. Also, $|diagonals(w)| = $ number of diagonals in $\Dg$ that are incident to $w$. Each of these diagonals is the diagonal of a unique quadrilateral in $T'$. After deletion of the diagonal, each of these quadrilaterals becomes a tile in the snake graph $\GTg$, and since $w$ cannot be on the shared (i.e. glued) edge of any two of these quadrilaterals, the vertex corresponding to $w$ in one tile does not coincide with the vertex corresponding to $w$ in another tile when the tiles are glued together to form the snake graph. Since all the vertices corresponding to $w$ remain distinct when $\GTg$ is formed, there are exactly $|diagonals(w)|$ vertices in $\GTg$ that correspond to $w$. In this way, $w$ corresponds to an entire equivalence class $[v]$ of vertices of $\GTg$, and this equivalence class has cardinality $|diagonals(w)|$, as desired. Let $m$ be the label of an edge in $T'$ incident to $w$. There are at least 2 quadrilaterals in $T'$ incident to $w$, so $m$ must be the side of some quadrilateral, hence it becomes an edge in $\GTg$ that is incident to some vertex in $\GTg$ that corresponds to $w$. Conversely, if $m$ is the label of an edge in $\GTg$ incident to a vertex of $\GTg$ that corresponds to $w$, then it is the side of a quadrilateral in $T'$ with a vertex at $w$, hence it is the label of an edge in $T'$ incident to $w$. So indeed the labels of edges incident to $[v]$ in $\GTg$ are the same as the labels of edges incident to the corresponding $w$ in $T'$.

We have considered both cases, so (ii) is indeed equivalent to (iv). Thus (i)-(ii) is equivalent to (iii)-(iv) for the case $|\Dg| \geq 2$. 

If $|\Dg| = 1$, then $T'$ is a triangulated quadrilateral. 
We can label the lone crossed diagonal 1, and the four sides 2, 3, 4, and 5, counterclockwise, such that the vertices are incident to edges $\{1,2,5\}$, $\{2,3\}$, $\{1,3,4\}$, and $\{4,5\}$. 
Then the equations given by (i)-(iv) are as follows:\\
(i) $x_1 = 0$ \\ 
(ii) $x_2+x_5=1; x_2+x_3 = 1; x_3+x_4 = 1; x_4+x_5 = 1$ \\
(iii) [none] \\
(iv) $x_1+x_2+x_5=1; x_2+x_3 = 1; x_1+x_3+x_4 = 1; x_4+x_5 = 1$\\
Here, the set of equations (i)-(ii) clearly implies (iii)-(iv), and in fact (iii)-(iv) implies (i)-(ii) as well (alternately add and subtract the equations in (iv) for example). Thus the equations (i)-(ii) is equivalent to the set of equations (iii)-(iv).
\end{proof}

\begin{proposition} \label{12affHull}
The set of equations {\rm(i)-(ii)} describes the affine hull of $\PG.$ 
\end{proposition}

\begin{proof}
We will define a new graph as well as a projection map.

Say the edge labels of $\GTg$ are $1..k$. Without loss of generality, say that edge labels $\{1,...,r\}$ occur on two distinct edges, and edge labels $\{r+1,...,k\}$ appear uniquely. (Note that by Remark \ref{1or2edgesSameLabel}, the same edge label cannot appear more than twice in $\GTg$, so those are the only two possibilities.)

Define a graph $G$ to be the graph $\GTg$, but with edges labeled in the following way. For each edge label $e$ that is unique in $\GTg$, label the corresponding edge of $G$ with that same label $e$. For each edge label $e$ that is found on two distinct edges in $\GTg$, label the corresponding edges of $G$ as $e$ and $k+e$.  Now $G$ is a graph that has the same structure as $\GTg$, but with unique edge labels. Therefore Lemma \ref{LP734} applies to $G$.

Define a projection $\pi: \mathbb{R}^{k+r} \rightarrow \mathbb{R}^k$ by $$(x_1,...,x_{k+r}) \mapsto (x_1+x_{k+1},...,x_r+x_{k+r},x_{r+1},...x_k).$$ This projection maps edge weight vectors of $G$ onto the corresponding edge weight vectors of $\GTg$, based on the above labeling scheme. Note that the projection respects the equivalence relation on vertices and edges of $\GTg$ defined earlier. Also note that the polytope $\PG$ is simply the image of $PM(G)$ under the map $\pi$. 

Now we can address our equations (i) and (ii). Note that the cluster expansion formula in Proposition \ref{MSmatching} is based on $\GTg$, so any edge $e$ that is not in $\GTg$ does not appear in the cluster expansion. Hence every such $x_e$ must be 0, so the equations (i) are true for $\PG$. For (ii), note that the polytope $\PG$ is simply the image of $PM(G)$ under the map $\pi$. Adding together equations in Lemma \ref{LP734}(ii') that correspond to equivalent vertices and then applying $\pi$, we get our equations (ii), so our equations (ii) are clearly true for $\PG.$ We now need only show that the equations (i)-(ii) are also sufficient (i.e. no others are needed). 

By Lemma \ref{LP734},  $dim \; PM(G) = |E(G)| - |V(G)| + 1.$ By the construction of the snake graph, this quantity is exactly the number of boxes in the graph, which is $d$. (Induction: the graph $G$ has the same structure as $\GTg$, which consists of $d$ boxes glued together. With 1 box, the formula from Lemma \ref{LP734} gives $dim \; PM(G) = 1.$ Each added box increases the number of vertices by 2 and the number of edges by 3, so $|E(G)| - |V(G)| + 1$ (the dimension) increases by 1. By induction, $dim \; PM(G) = d$.) Recall that in our case of the polygon, $d = |\Dg|$, so we have $dim \; PM(G) = |\Dg|$.

Recall that $\PG$ was defined as a polytope in $\R^{2n+3}$. From Lemmas \ref{34indep} and \ref{12equiv34}, we already know that (i)-(ii) consists of (or is equivalent to) $2n+3-|\Dg|$ linearly independent hyperplanes, each of which cuts down (by 1) the dimension of the affine subspace that $\PG$ lives in. Subtracting, we see that no more equations are needed if and only if the dimension of the affine hull of $\PG$ is exactly $|\Dg|$.

So we need to show that the projection map $\pi$ preserves the dimension of $PM(G).$ Instead of working directly with the affine hull of $PM(G)$ (an affine subspace of $\mathbb{R}^{3n+1}$), we will work instead with $Q$, the linear subspace that is given by writing the equations in Lemma \ref{LP734}(ii') with a 0 on the right-hand side instead of a 1. So $Q$ is just the affine hull of $PM(G)$, but shifted so it passes through the origin. We will show that $dim \; \pi(Q) = dim \; Q.$ (It suffices to show this instead because if $T$ is any linear transformation of a space $S$, then $T(S+\overrightarrow{OP}) = T(S)+\overrightarrow{O \; T(P)}$, which clearly has the same dimension as $T(S)$.)

First we will define a basis for $Q.$ We will imagine an element of $Q$ as an assignment of weights to the edges of the graph $G$ such that the sum of the weights of edges incident to each vertex is 0. By the above paragraph, we need to find $|\Dg|$ linearly independent vectors in $Q$. Define $q_i$ to be the following assignment of edge weights: weight 1 on the horizontal edges of the $i^{th}$ box in the snake graph, weight $-1$ on the vertical edges of the $i^{th}$ box, and weight 0 on all other edges (see figure below for what $q_2$ looks like in our example). Since there are $|\Dg|$ boxes in $G$, and these edge weights are linearly independent by construction, the $q_i$ form a basis for $Q$. 

Next we will define a basis for $ker \; \pi$. Again, we imagine a vector here as an assignment of weights to the edges of the graph $G$. For $1 \leq i \leq r,$ define $b_i$ to be following assignment of edge weights: weight 1 on the edge labeled $i$, weight $-1$ on the edge labeled $k+i$, and weight 0 on all other edges. See the figure below for what $b_4$ looks like in our example (note: here, the basis is $\{b_3, b_4, b_5\}$ rather than ``$\{b_1, b_2, b_3\}$" because we did not relabel the non-unique edges, but it doesn't matter). This is well-defined since at most two edges collapse to the same label under $\pi$, and the pairs of edges that do collapse are precisely the ones involved here. The $b_i$ here clearly form a basis for $ker \; \pi$. 

\begin{figure}[h!]
        \centering
        \begin{subfigure}[b]{0.46\textwidth}
                \centering
                \includegraphics[width=\textwidth]{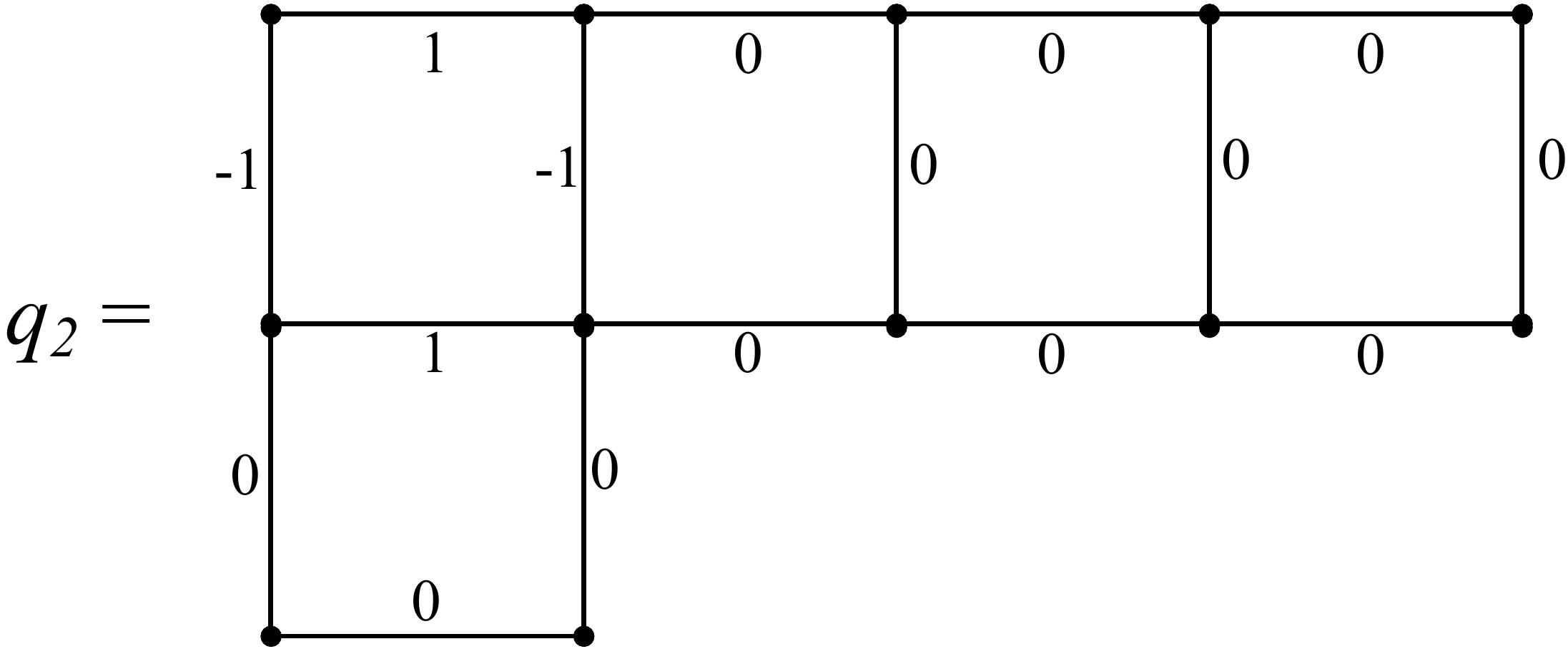}
                \label{fig:qi}
        \end{subfigure}
        \qquad
        \begin{subfigure}[b]{0.46\textwidth}
                \centering
                \includegraphics[width=\textwidth]{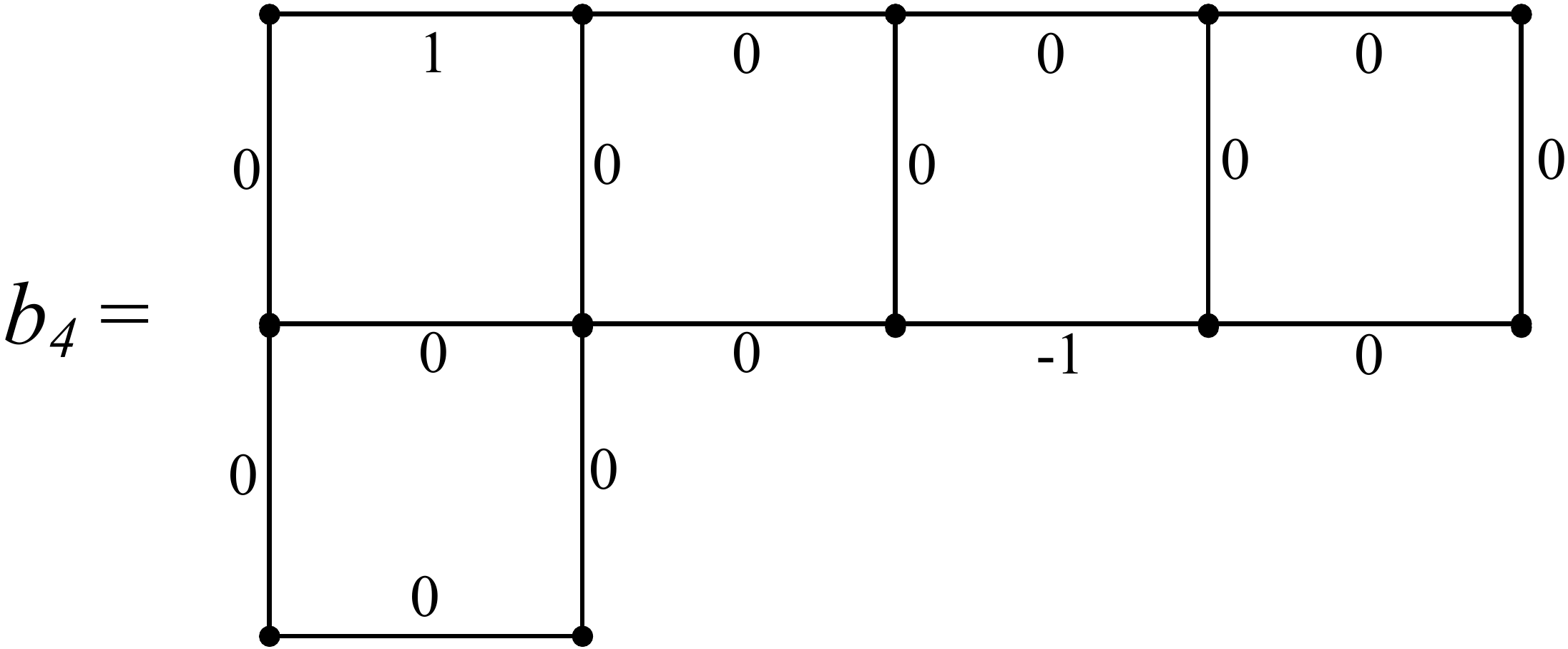}
                \label{fig:bi}
        \end{subfigure}
        \vspace{7pt}
\end{figure}

Now we will show that the $b_i$ (basis vectors of $ker \; \pi$) are all linearly independent from the $q_i$ (basis vectors of $Q$). Suppose a linear combination $\sum_i c_i q_i + d_i b_i = 0$. Note that all interior edges on the $b_i$ have weight 0. The leftmost edge of box 1 is nonzero only in $q_1$, forcing $c_1 = 0$. The edge of box 2 that touches box 1 is nonzero only in $q_2$ (forcing $c_2=0$), or in $q_1$ and $q_2$ (forcing $c_2 = -c_1 = 0$). Either way, $c_2 = 0$. Similarly, the edge of box 3 that touches box 2 is nonzero only in $q_3$, or in $q_2$ and $q_3$, forcing $c_3 = 0$. Continuing in this way, we see that all $c_i = 0.$ Now, the $b_i$ are already linearly independent, so all $d_i$ must be 0 as well. Therefore, we have shown that the basis vectors for $ker \; \pi$ are indeed linearly independent from the basis vectors of $Q$. From this we can conclude that $dim \; \pi(Q) = dim \; Q.$  

Thus $dim \; \PG = dim \; \pi(PM(G)) =  dim \; \pi(Q) = dim \; Q = dim \; PM(G) = |\Dg|,$ so no more equations are needed and we are done.
\end{proof}

\begin{corollary} \label{dimPGa}
$dim \; \NTg =\; dim \; \PG = |\Dg|$.
\end{corollary}

\begin{proof}
The dimension of $\PG$ was discussed in the proof of Lemma \ref{12affHull}. The equality with $dim \; \NTg$ simply follows from Remark \ref{PGaNewtPolyOfNumerator}.
\end{proof}

\begin{example}{\rm
Using the same example as in section \ref{Intro}, we will illustrate Proposition \ref{PGaAffHull} and Corollary \ref{dimPGa}. Our computations were made or checked with the help of the software {\tt polymake} \cite{Polymake} and {\tt Mathematica}.
\begin{itemize}
\item (i)-(ii): Since edges 14 and 15 of $T$ do not appear in $\GTg$, the affine hull of $\PG$ includes the equations $x_{14} = 0$ and $x_{15} = 0$. The other equations defining the affine hull of $\PG$ come from each of the equivalence classes of vertices of $\GTg$. For example, the equivalence class $\{a,a',a''\}$ has cardinality 3, and the edges of $\GTg$ incident to those vertices are labeled 1, 2, 3, 4, and 7, so we get the equation $x_1+x_2+x_3+x_4+x_7 = 3.$ 
\item (iii)-(iv): We can get the same list of equations as above by using $T$ directly. Since edges 14 and 15 of $T$ do not appear in $T'$, we again get the equations $x_{14} = 0$ and $x_{15} = 0$. 
The other equations defining the affine hull of $\PG$ come from each vertex of $T'$. For example, the edges incident to vertex $A$ are $\{1,2,3,4,7\}$, three of which ($\{2,3,4\}$) are in $\cc \cup \Dg$, so we get the equation $x_1+x_2+x_3+x_4+x_7 = 3.$ 
\item Since we are triangulating a 9-gon, $n = 6$. Also, $\Dg = \{2,3,4,5,6\}$, so $dim \; \NTg = |\Dg| = 5$. The number of equations defining the affine hull is $2n+3 - |\Dg| = 10$.
\end{itemize}
So here are the 10 equations for the affine hull of $\PG$ (using either (i)-(ii) or (iii)-(iv)). The reader may check that they form a linearly independent set.
\begin{align*}
\text{absent edge 14:} \quad & x_{14} = 0 \\
\text{absent edge 15:} \quad & x_{15} = 0 \\
\{a,a',a''\} \text{ or } A: \quad & x_1+x_2+x_3+x_4+x_7 = 3 \\
\{b,b'\} \text{ or } B: \quad & x_5+x_6+x_7+x_8 = 2 \\
\{c\} \text{ or } C: \quad & x_8+x_9 = 1 \\
\{d\} \text{ or } D: \quad & x_6+x_9+x_{10} = 1 \\
\{e,e'\} \text{ or } E: \quad & x_4+x_5+x_{10}+x_{11} = 2 \\
\{f\} \text{ or } F: \quad & x_3+x_{11}+x_{12} = 1 \\
\{g\} \text{ or } G: \quad & x_2+x_{12}+x_{13} = 1 \\
\{h\} \text{ or } H: \quad & x_1+x_{13} = 1 \\
\end{align*}}
\end{example}

We now need a classic result from the literature on bipartite graphs and associated polytopes.

\begin{lemma} \label{BilleraSaran}
(\cite{BS} 2.1): Denote the complete bipartite graph by $K_{n,n}.$ The face lattice of the Birkhoff polytope $PM(K_{n,n})$ is isomorphic to the lattice of elementary subgraphs of $K_{n,n}$ ordered by inclusion.
\end{lemma}

\begin{remark} \label{refinementBilleraSaran}
The isomorphism of lattices described in Lemma \ref{BilleraSaran} still holds when $K_{n,n}$ is replaced with any elementary subgraph of $K_{n,n}$. Specifically, if $H$ is an elementary subgraph of $K_{n,n}$, then under this isomorphism it corresponds to a face $F$ of $K_{n,n}$, and any elementary subgraph of $H$ corresponds to a face of $F$, so the face lattice of $F$ is isomorphic to the lattice of elementary subgraphs of $H$. Also note that our graph $G$ is an elementary subgraph of $K_{n,n}$, and $PM(G)$ is a face of $K_{n,n}$. 
\end{remark}

\begin{theorem} \label{latticesIsom}
The face lattice of $\NTg$ (and of $\PG$) is isomorphic to the lattice of all elementary subgraphs of $\GTg$, ordered by inclusion.
\end{theorem}

\begin{proof}
Note that $\PG$ is the image of $PM(G)$ under the projection map $\pi$ defined in the proof of Proposition \ref{12affHull}. Since the two polytopes have the same dimension, $|\Dg|$, and $\pi$ is a linear transformation, $\PG$ must be combinatorially isomorphic to $PM(G)$. By Remark \ref{PGaNewtPolyOfNumerator}, $\NTg$ is a translate of $\PG$, so the face lattice of $\NTg$ must be isomorphic to the face lattice of $PM(G)$. But $G$ is an elementary subgraph of $K_{n,n}$, so by Remark \ref{refinementBilleraSaran}, the face lattice of $PM(G)$ is isomorphic to the lattice of elementary subgraphs of $G$, which is clearly isomorphic to the lattice of elementary subgraphs of $\GTg$, and we are done. 
\end{proof}

\begin{corollary} \label{monomCorrToVert}
The following are in one-to-one correspondence: 
\begin{align*}
{\rm(i)}& \quad \text{Laurent monomials in the $T$-expansion of } x_\cc \\ 
{\rm(ii)}& \quad \text{perfect matchings of } \GTg\\
{\rm(iii)}& \quad \text{vertices of } \NTg
\end{align*}
\end{corollary}

\begin{proof}
By Proposition \ref{MSmatching}, the numerator of every Laurent monomial in a given cluster expansion corresponds to a perfect matching of $\GTg$, which is an atom in the elementary subgraph lattice described in Theorem \ref{latticesIsom}, corresponding to an atom in the face lattice of the polytope $\NTg$, i.e. a vertex.
\end{proof}

\begin{example}{\rm
Using the same example as in section \ref{Intro}, we will illustrate Corollary \ref{monomCorrToVert}.
One Laurent monomial in the $T$-expansion of $x_\cc$ can be written as $\frac{x_1 x_2 x_3 x_5^2 x_9}{x_2 x_3 x_4 x_5 x_6}$. \\
This corresponds to the $\NTg$ vertex $(1,0,0,-1,1,-1,0,0,1,0,0,0,0,0,0)$ and the perfect matching here:}

\begin{figure}[h!]
\centering
\includegraphics[scale=0.3]{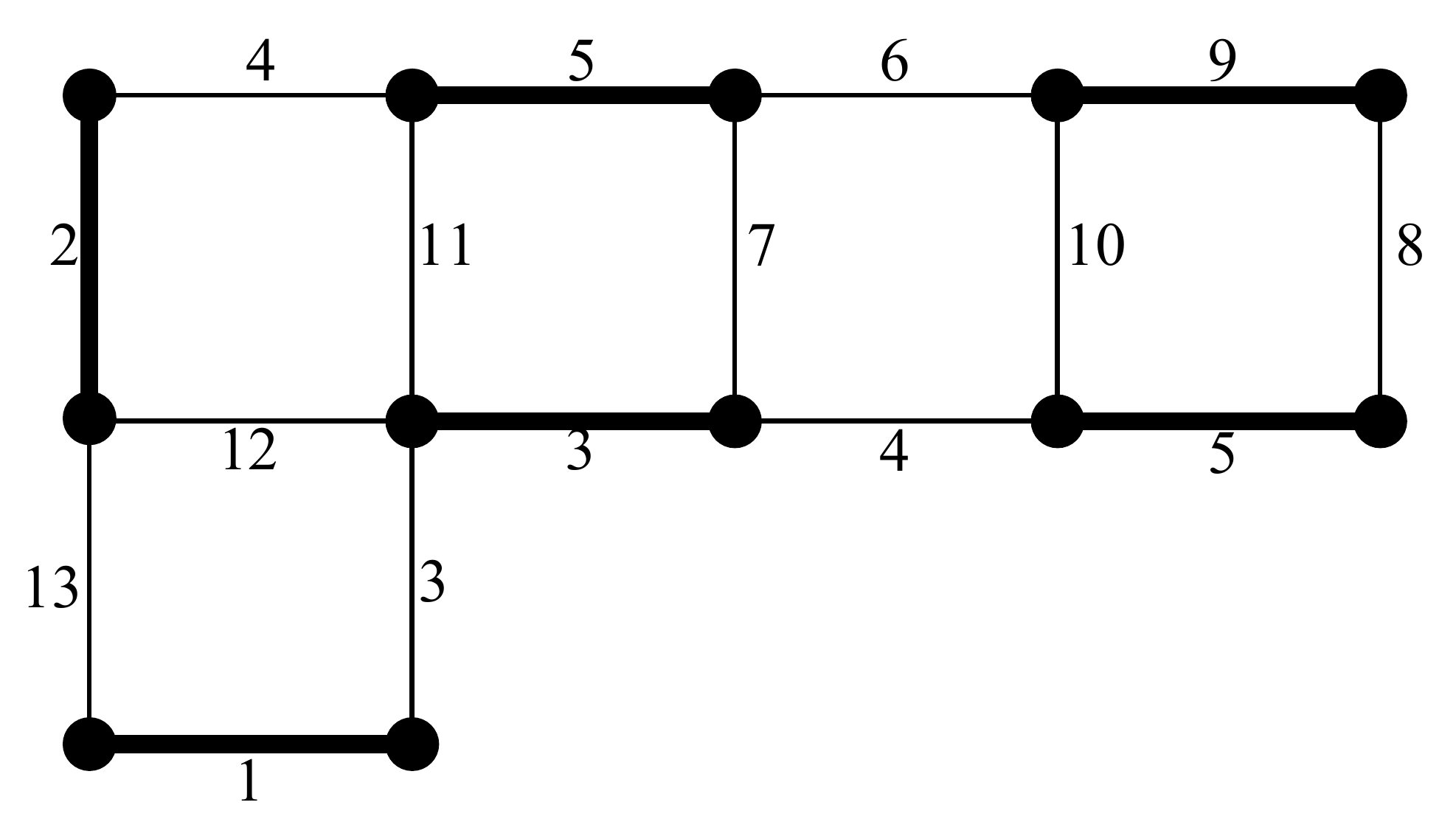}
\label{fig:matching}
\end{figure}
\end{example}

The lattices described in Theorem \ref{latticesIsom} are graded: the rank of a face is 1 more than its dimension, and the rank of an elementary subgraph is 1 more than the number of chordless cycles it contains. So the $d$-faces of $\NTg$ are in bijection with elementary subgraphs of $\GTg$ containing exactly $d$ chordless cycles (these cycles may or may not be disjoint).

In particular, let $P(i)$ be the perfect matching of $\GTg$ that corresponds to vertex $i$ of the polytope. Given a set of vertices $(i_1,\ldots,i_r)$ that make up a face, the corresponding elementary subgraph is obtained by superimposing $P(i_1),\ldots,P(i_r).$ Conversely, given an elementary subgraph $H$, if the set of all perfect matchings of $\GTg$ that lie entirely on $H$ is $P(i_1),\ldots,P(i_r)$, then $(i_1,\ldots,i_r)$ is the set of vertices making up the corresponding face. Note that to count only the number of vertices making up a face, we need only count the number of perfect matchings of the ``cycle part" of $H$ (i.e. the union of all the cycles in $H$), since the rest of $H$ is already matched. 

Also, in particular, the facets of $\NTg$ are the $(n-1)$-faces, so they can be found by finding the elementary subgraphs of $\GTg$ containing $n-1$ chordless cycles. We will do this below.

\begin{example}{\rm
We will illustrate Theorem \ref{latticesIsom} using a small example. Let $T$ be a triangulation of a pentagon, with $\cc$ a diagonal that crosses both diagonals of $T$. Then $\GTg$ consists of two boxes, and the face lattice of $\NTg$ is isomorphic to the following lattice of elementary subgraphs of $\GTg$.

\begin{figure}[h!]
\centering
\includegraphics[scale=0.55]{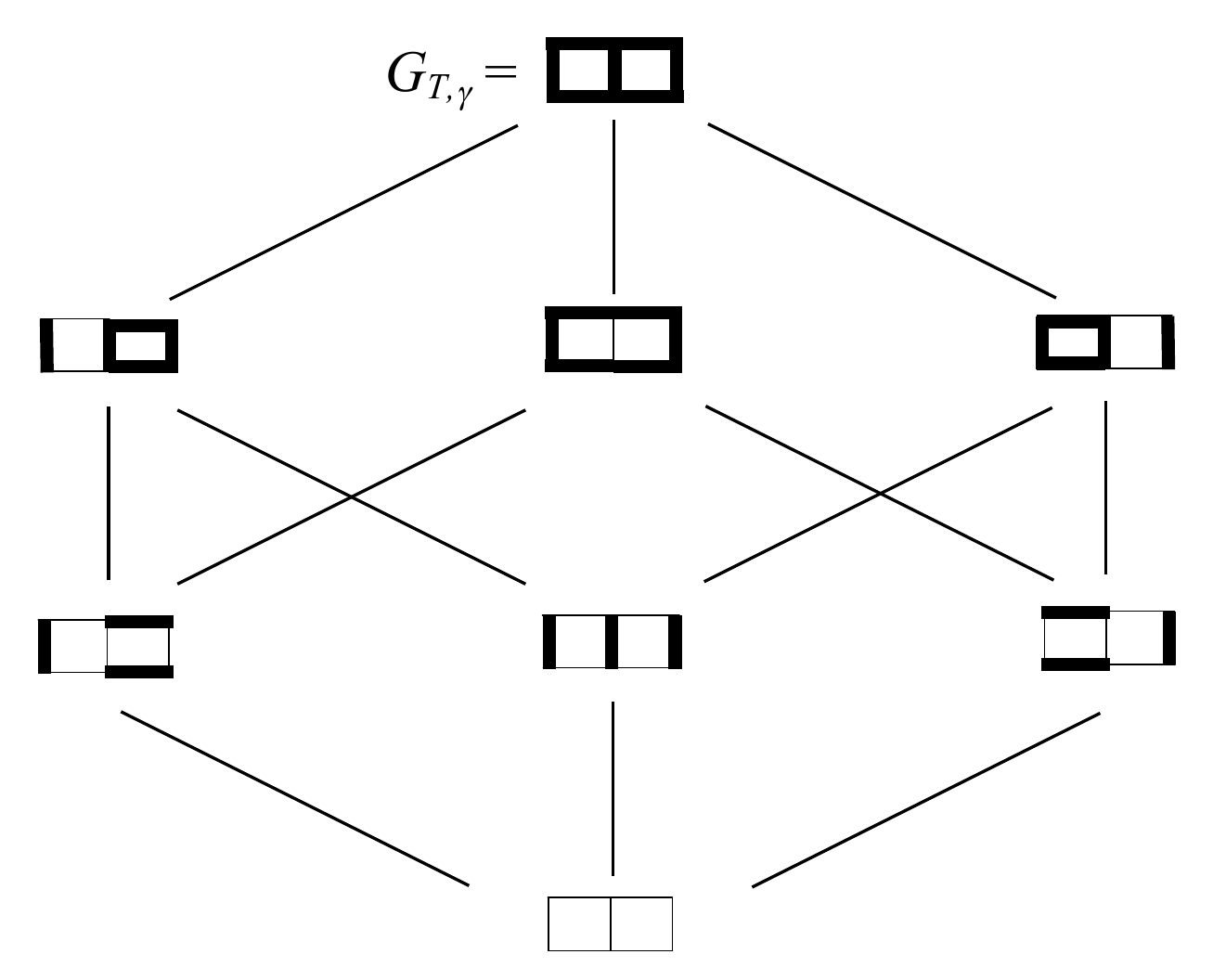}
\caption{Lattice of Elementary Subgraphs of $\GTg$}
\label{fig:elemsubgrlattice}
\end{figure}

In this example, the dimension of $\NTg$ is $|\Dg| = 2$, and $\NTg$ is a triangle. Also note that the length of every maximal chain in this lattice is 3.} 
\end{example}

\begin{example}{\rm
Theorem \ref{latticesIsom} can be used to find the $f$-vector of $\NTg$, by counting how many elementary subgraphs contain $d$ chordless cycles for each $d$ from 0 to $|\Dg|-1$. For example, the $f$-vector of $\NTg$ from our original example in Figure \ref{fig:example} is $(11,31,39,25,8)$.
}\end{example}

We now turn our attention to the facets of our Newton polytopes.

\begin{proposition} \label{facetsPGaFromGa}
For any diagonal $\cc$, the facets of the polytope $\PG$ can be found from the snake graph $\GTg$ by writing the following inequalities:
\begin{align*}
{\rm(i)}\quad & x_e \geq 0 \text { for each } e \in \GTg \text { such that e is an interior edge of } \GTg.\\
{\rm(ii)}\quad & x_e \geq 0 \text { for each pair of opposite exterior edges } \{e,f\} \text{ of $\GTg$ such that at least one }\\
& \text {of the two edges has a label e that is unique in $\GTg$. (see figure below)} \\
{\rm(iii)}\quad & x_a+x_b+x_c \leq 2 \text{ for each pair of opposite exterior edges } \{e,f\} \text { of $\GTg$ that includes }\\ 
& \text {no unique labels, where a, b and c are the labels of edges shown in the figure below.}
\end{align*}

\begin{figure}[h!]
\centering
\includegraphics[scale=0.4]{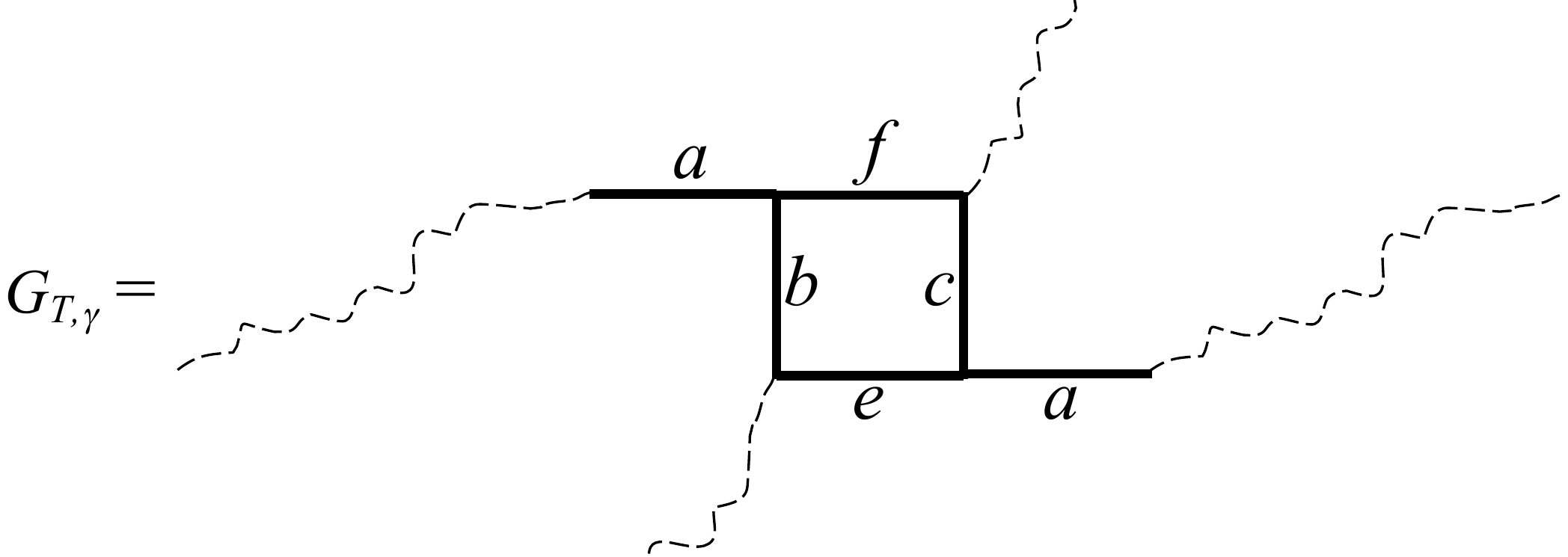}
\label{fig:facetlemmafig1}
\end{figure}

\end{proposition}

\begin{remark} \label{facetPGaDetails}
Details of Proposition \ref{facetsPGaFromGa}:
\begin{itemize}
\item In (ii), if both of the opposite exterior edges are unique labels in $\GTg$ (only possible if $T'$ is a triangulated quadrilateral or triangulated pentagon), arbitrarily choose one to be $e$. 
\item The pair of edges $\{e,f\}$ in (ii) happens precisely on the first, second, penultimate, and last tile of $\GTg$ (the ones of these that are not corners). The pair of edges $\{e,f\}$ in (iii) happens on all other tiles in $\GTg$ that are not corners. This follows from Lemma \ref{SaFeatVsGaFeat}. 
\item In the figure above, the orientation of the pair of edges $\{e,f\}$ may be vertical, not horizontal, as below, and the lemma continues to hold.
\end{itemize}
\end{remark}

\begin{figure}[h!]
\centering
\includegraphics[scale=0.4]{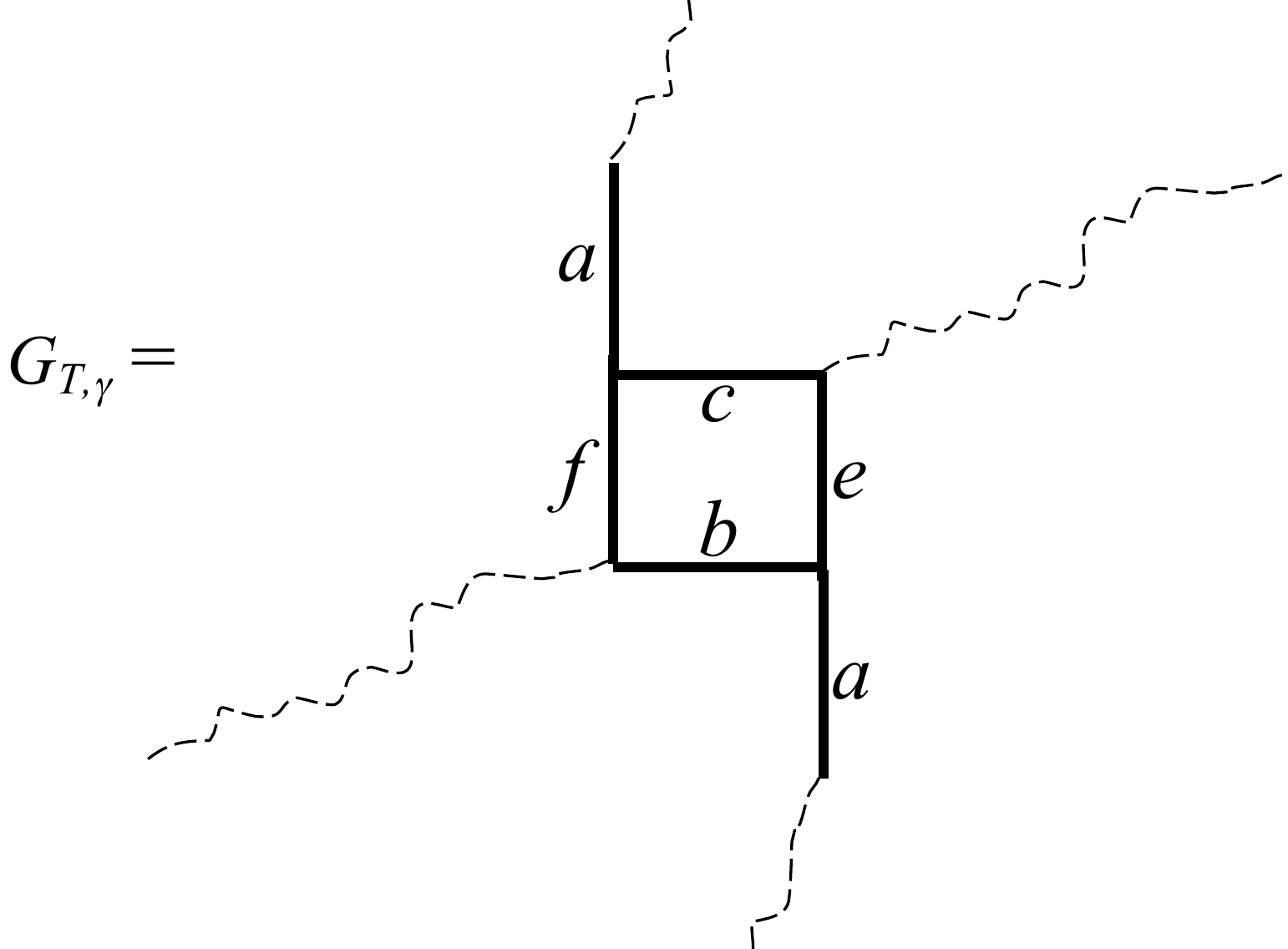}
\label{fig:facetlemmafig2}
\end{figure}

\pagebreak

\begin{proof}

Let $F(\GTg)$ be the set of all elementary subgraphs of $\GTg$ containing $n-1$ chordless cycles. We know from the comments after the proof of Theorem \ref{latticesIsom} that each facet of $\PG$ corresponds to an elementary subgraph in $F(\GTg)$. Every such subgraph $H \in F(\GTg)$ falls into one of 3 cases, which will become (i), (ii), and (iii):\\

Case 1: 

\begin{figure}[h!]
\centering
\includegraphics[scale=0.4]{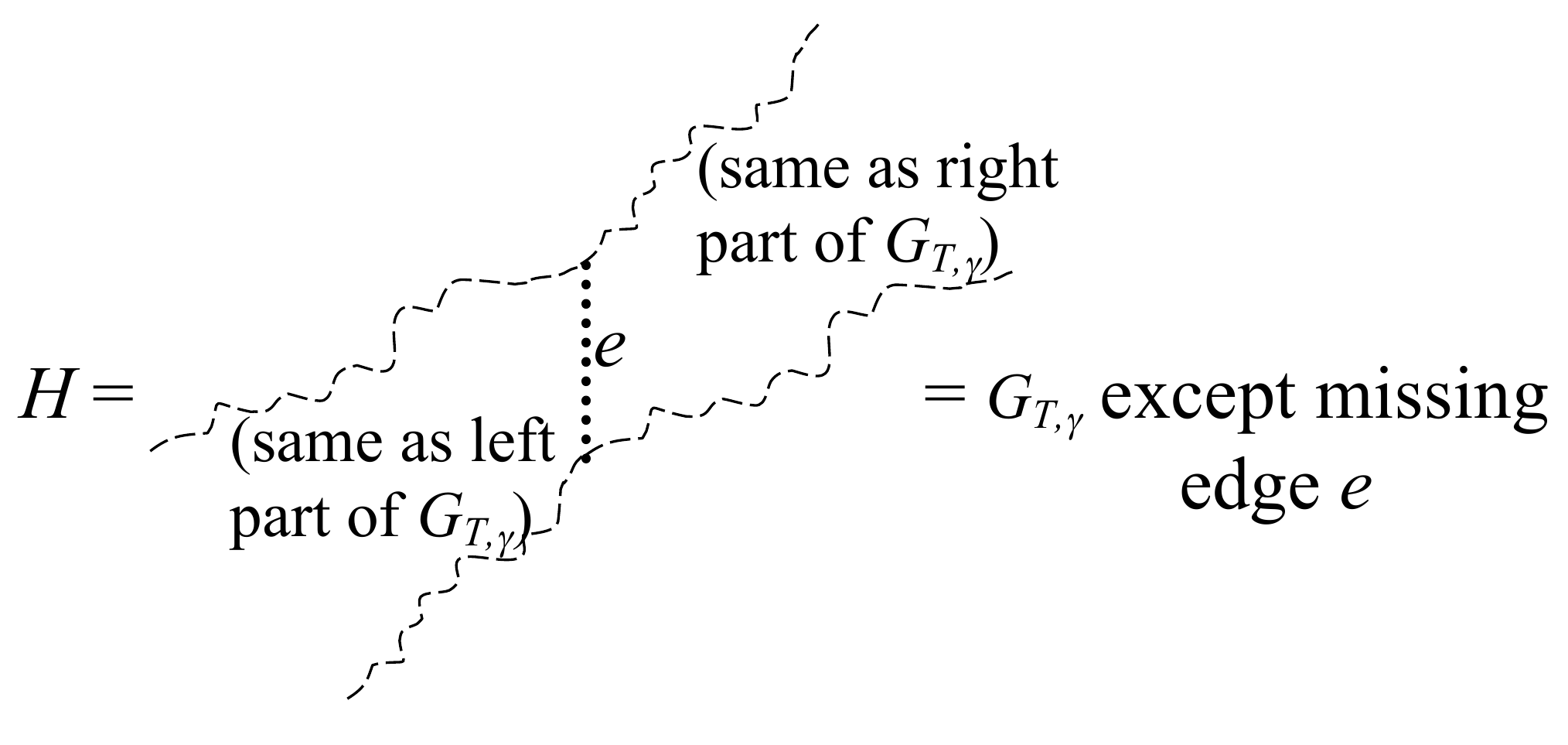}
\label{fig:case1}
\end{figure}

In Case 1, $H$ is just $\GTg$, but missing exactly 1 edge. It must be an interior edge, and interior edge labels in snake graphs are unique (see Lemma \ref{SaFeatVsGaFeat}). Say the label is $e$. (The edge $e$ may be horizontally oriented, as in the figure, or it may be vertically oriented - it doesn't matter for this argument.) Then every perfect matching of $\GTg$ that is a subgraph of $H$ has no edge labeled $e$, so the characteristic vector of the vertex corresponding to such a perfect matching has $x_e = 0$. Conversely, every perfect matching of $\GTg$ that is not a subgraph of $H$ must include the edge labeled $e$ (since $H$ is just $\GTg$, but missing that one edge), so the characteristic vector of the vertex corresponding to such a matching does not have $x_e = 0$, but rather $x_e=1$. Thus the hyperplane $x_e = 0$ is both necessary and sufficient - it includes all the desired vertices and no others. This confirms (i).\\

Case 2: 

\begin{figure}[h!]
\centering
\includegraphics[scale=0.4]{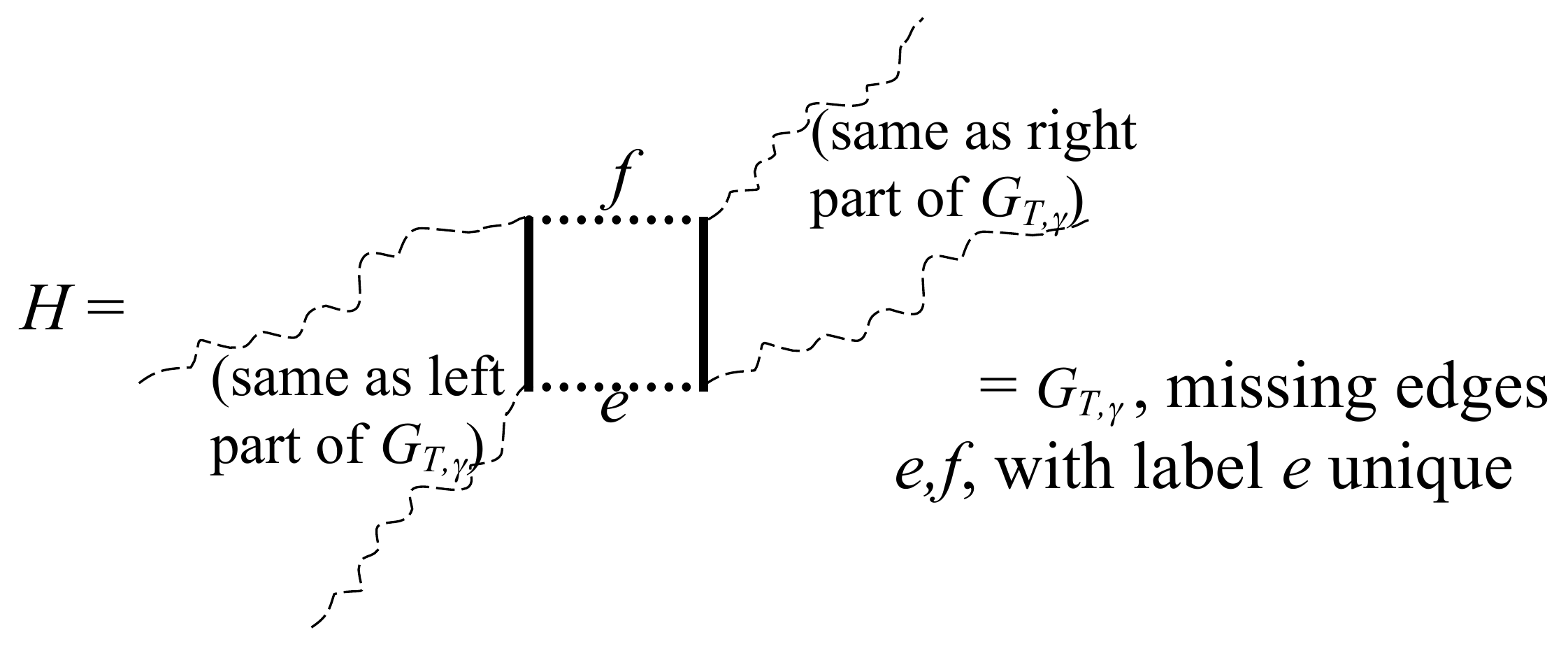}
\label{fig:case2}
\end{figure}

In Case 2, $H$ is just $\GTg$, but missing exactly 2 exterior edges opposite each other, at least one of which has a label that is unique in $\GTg$. Say this unique label is $e$. (If both edge labels appear uniquely in $\GTg$, choose either.) So $H$ is the union of two disjoint snake graphs, and the two ``missing" edges may be horizontal, as in the figure, or they may be vertical. Furthermore, it doesn't matter whether the uniquely labeled edge $e$ happens to be the top or the bottom ``missing" edge. As in Case 1, every perfect matching of $\GTg$ that is a subgraph of $H$ has no edge labeled $e$, so the characteristic vector of the vertex corresponding to such a perfect matching has $x_e = 0$. Conversely, every perfect matching of $\GTg$ that is not a subgraph of $H$ must include the edge labeled $e$, or the edge opposite to it. But opposite exterior edges of a perfect matching of a snake graph agree (either both present or both absent), so since such a matching includes either, it must include both. Thus the characteristic vector of the vertex corresponding to such a matching does not have $x_e = 0$, but rather $x_e=1$. So again, the hyperplane $x_e = 0$ is both necessary and sufficient - it includes all the desired vertices and no others. This confirms (ii).\\

Case 3: $H$ is just $\GTg$, but missing exactly 2 exterior edges opposite each other, neither of which has a label that is unique in $\GTg$. Then, as in Case 2, $H$ is the union of two disjoint snake graphs, and without loss of generality (i.e. the two parts may be vertically oriented for example, but it doesn't matter), looks like the figure below.

\begin{figure}[h!]
\centering
\includegraphics[scale=0.4]{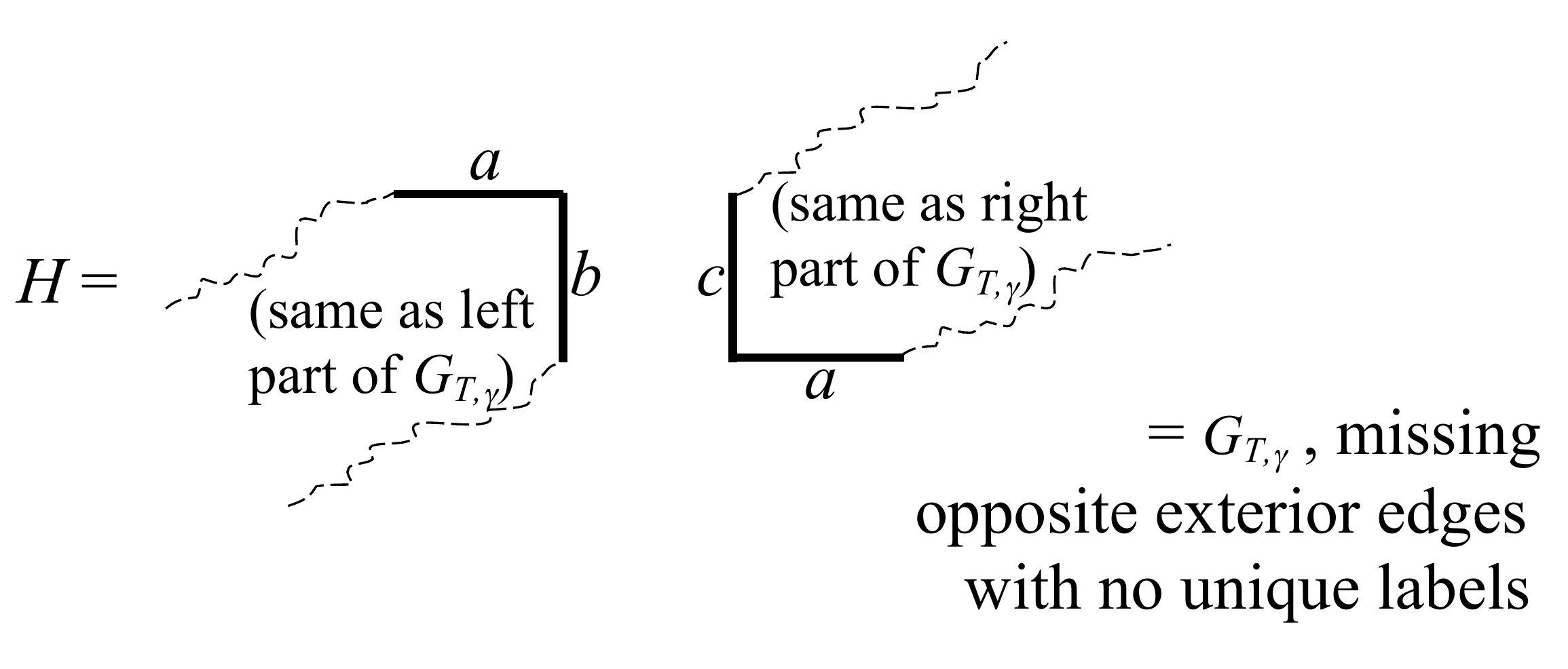}
\label{fig:case3}
\end{figure}

Note that the construction of snake graphs (or the comments about vertex equivalence earlier in this paper) guarantees the two edges labeled $a$ do indeed have the same label, and that no other edges of $\GTg$ are labeled $a$. Also note that the edges labeled $b$ and $c$ are interior edges of $\GTg$, so they are uniquely labeled. Every perfect matching of $\GTg$ that is a subgraph of $H$ looks locally like one of four types: 

\pagebreak

\begin{figure}[h!]
\centering
\includegraphics[scale=0.4]{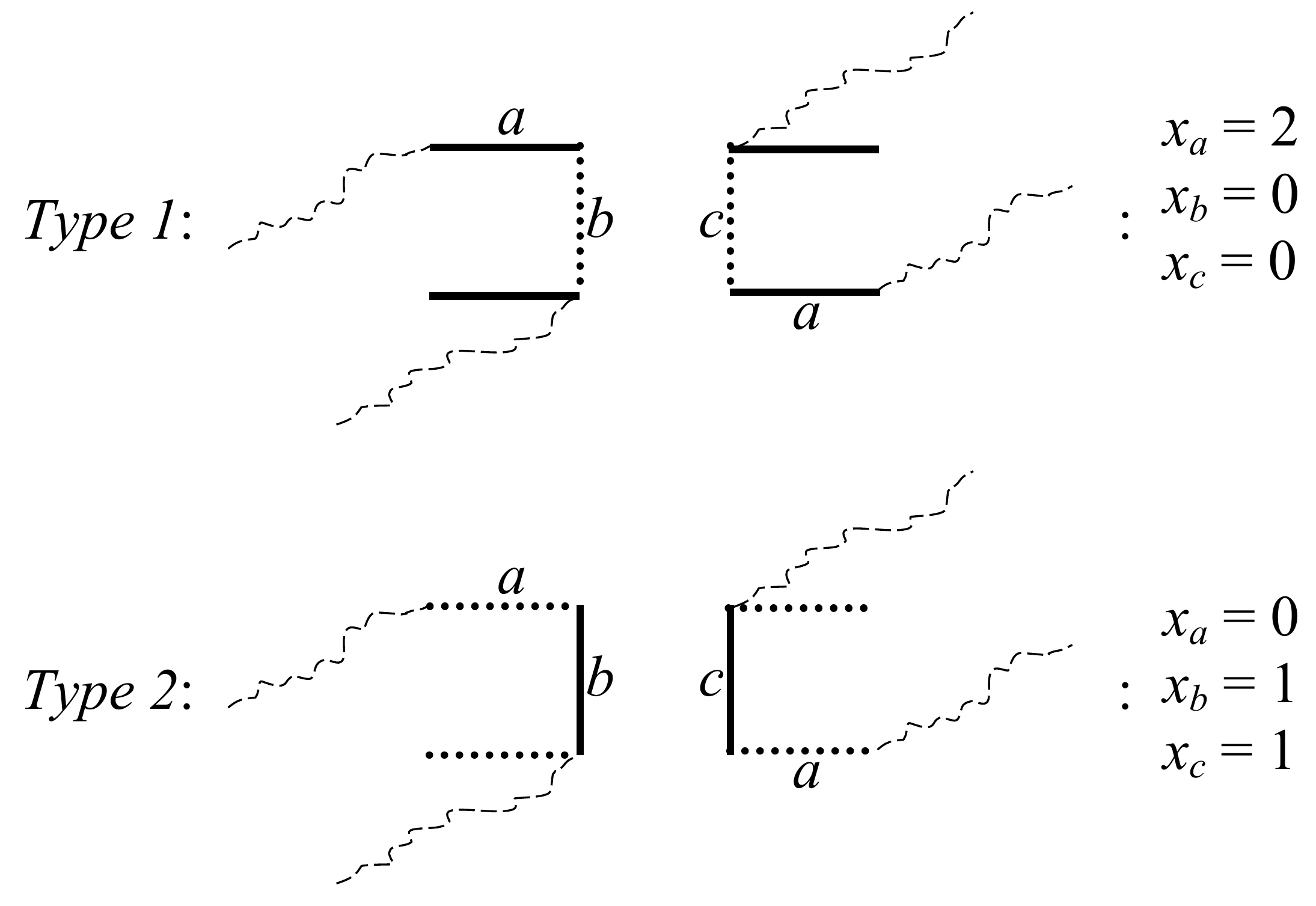}
\label{fig:fourtypes1}
\includegraphics[scale=0.4]{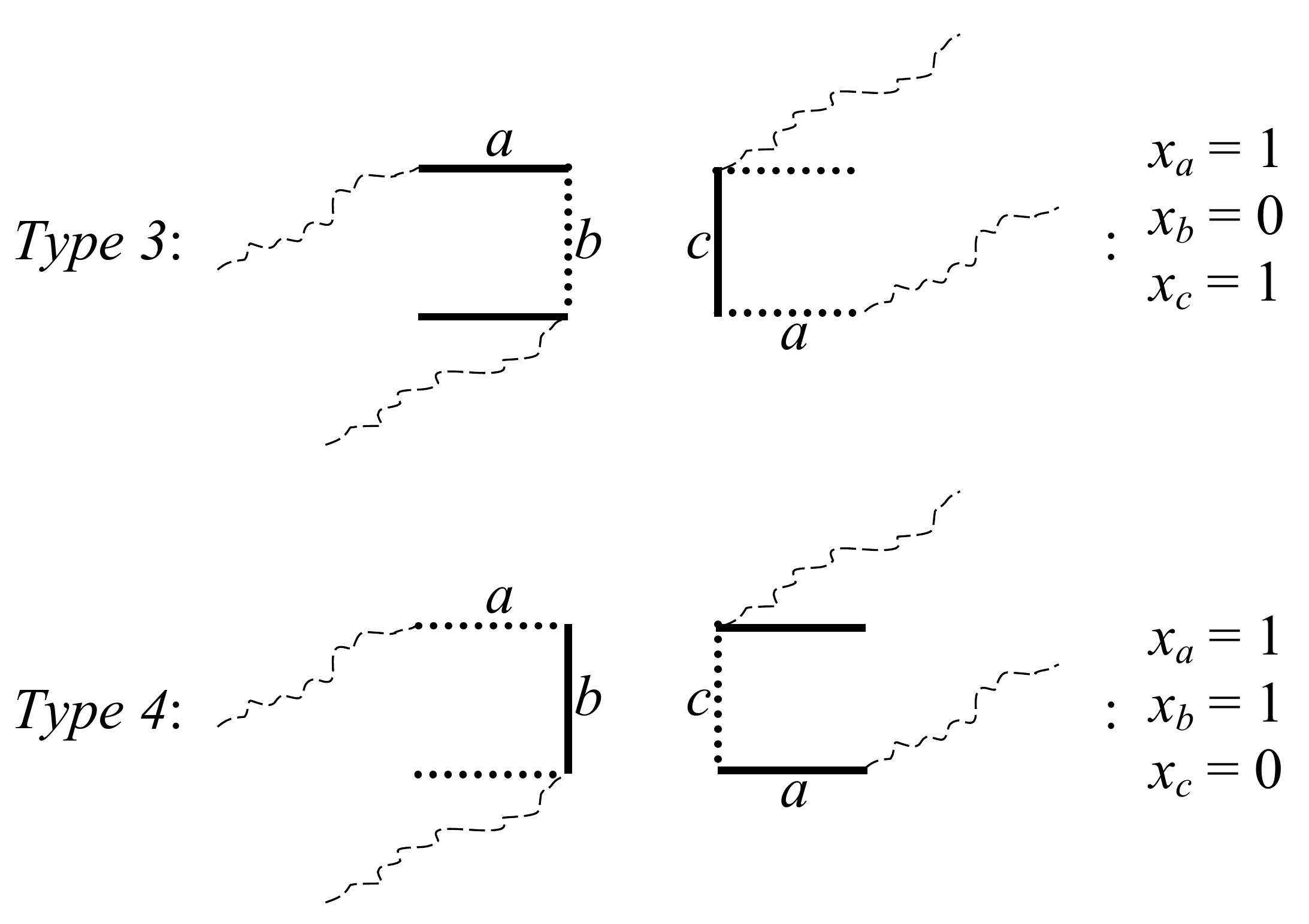}
\label{fig:fourtypes2}
\end{figure}

In each of these 4 cases, it is true that $x_a+x_b+x_c = 2$. Conversely, every perfect matching of $\GTg$ that is not a subgraph of $H$ must include both edges of $\GTg \backslash H$, so it cannot include any edge labeled $a, b,$ or $c$, hence $x_a+x_b+x_c = 0$, not $2$. Thus the hyperplane $x_a+x_b+x_c = 2$ is both necessary and sufficient - it includes all the desired vertices and no others. This case proves that (iii) gives a facet, and the direction of inequality that defines the half-space is clear.
\end{proof}

\begin{corollary} \label{numFacets}
If $|\Dg| \geq 2$, the number of facets of $\NTg$ (or $\PG$) is $2d -1-t$, where $t$ is the number of corners in the snake graph, or equivalently, the number of imbalanced diagonals in $T'$. (If  $|\Dg| =1$, the polytope is a line segment, so it has 2 facets that are the endpoints.)
\end{corollary}

\begin{proof}
Assume $|\Dg| \geq 2$. First note that by Lemma \ref{SaFeatVsGaFeat}, the number of imbalanced diagonals in $T'$ is the number of corners in $\GTg$, so it is valid to call them both $t$. The number of interior edges of $\GTg$ is 1 less than the number of boxes, so Proposition \ref{facetsPGaFromGa}(i) gives $d-1$ facets. Combining Proposition \ref{facetsPGaFromGa}(ii)-(iii), we see that we get 1 facet for every pair of opposite exterior edges, and since $|\Dg| \geq 2$, there are at least two boxes in $\GTg$, so there is clearly 1 pair of opposite exterior edges for every box that is not at a corner in the graph. The number of such boxes not at a corner is $d -t$. Adding $d -1$ to $d -t$ gives the desired result. 
\end{proof}

\begin{example}{\rm
Using the same example as in Section \ref{Intro}, we will illustrate Proposition \ref{facetsPGaFromGa} and Corollary \ref{numFacets}. 
\begin{itemize}
\item (i): The interior edges of $\GTg$ are 12, 11, 7, and 10, giving the facet-defining inequalities $x_{12} \geq 0$, $x_{11} \geq 0$, $x_7 \geq 0$, and $x_{10} \geq 0$.
\item (ii): Checking the first, second, penultimate, and last box of $\GTg$, we get the pairs of opposite exterior edges $\{3,13\},$ [no pair], $\{4,6\}$, and $\{5,9\}$, respectively, with the unique labels in each pair being 13, 6, and 9, respectively. This gives the facets $x_{13} \geq 0$, $x_6 \geq 0$, and $x_9 \geq 0$.
\item (iii): Checking the remaining (i.e. third) box, we get the inequality $x_4+x_7+x_{11} \leq 2$.
\item Since there is 1 corner in $\GTg$ (the second box), or equivalently, 1 imbalanced diagonal in $T'$ (the edge labeled 3), we have $t = 1$. Again, $d = |\Dg| = 5$, so the number of facets is $2d -1-t = 8.$ This confirms that we have indeed found them all. Here they are in a list:
\end{itemize}
$$
x_{12} \geq 0,\; \; x_{11} \geq 0, \; \; x_7 \geq 0,\; \; x_{10} \geq 0, \; \; x_{13} \geq 0,\; \; x_6 \geq 0,\; \; x_9 \geq 0, \; \; x_4+x_7+x_{11} \leq 2
$$}
\end{example}

\begin{theorem} \label{NTgFromTg}
For any diagonal $\cc$, the polytope $\NTg$ can be found directly from $T$ as follows:\\
Affine hull equations:
\begin{align*}
{\rm(i)}& \quad \text{For each edge $e$ of } T \backslash T'  \text{, write } x_e = 0.\\
{\rm(ii)}& \quad \text{For each vertex $w \in T'$, write } \sum_{e \ni w} x_e = 1 \text{ if } w \in \cc, \text{ or write } \sum_{e \ni w} x_e =0 \text{ if } w \notin \cc.
\end{align*}
Facet-defining inequalities:
\begin{align*}
{\rm(iii)}& \quad \text{For every boundary segment $e \in T'$ not incident to } \cc \text{, write }  x_e \geq 0.\\
{\rm(iv)}& \quad \text{For every pair of boundary segments $\{b,c\}$ of $T'$ that are opposite sides}\\
& \quad \text{of $Q_{\tau_a}$, where $a \in \Dg$ is a balanced diagonal, let the other pair of opposite sides }\\
& \quad \text{of $Q_{\tau_a}$ be $\{e,f\}.$ Exactly one of these three cases will hold for each pair $\{e,f\}$:}\\
& \qquad \text{- If } \{e,f\} \subset \{\tau_{i_2},\ldots,\tau_{i_{d-1}}\} \text{, write the inequality } x_a+x_b+x_c \leq 1.\\  
& \qquad \text{- If one of } \{e,f\} \text{ (say $e$) is a boundary segment of } T', \text{ write } x_e \geq 0.\\ 
& \qquad \text{- Otherwise, write } x_e \geq -1, \text{ where $e$ is diagonal } \tau_{i_1} \text { or } \tau_{i_d}.
\end{align*}
\end{theorem}

\begin{proof}
This proof will follow from previous propositions and the shift of 1 unit downward in the direction of each of the crossed diagonals $\Dg$ described in Remark \ref{PGaNewtPolyOfNumerator}. Specifically, statement (i) was proven in Proposition \ref{PGaAffHull}, and is not affected by the shift. 

For statement (ii), we will shift the variables from Proposition \ref{PGaAffHull}(iv). The 1 unit downward translation means that each instance of $x_e$ should be replaced with $(x_e + 1)$ if $e \in \Dg$, and left as $x_e$ if $e \notin \Dg$. Modifying Proposition \ref{PGaAffHull}(iv) in this way, we get 
$$|\{e \in \Dg: e \ni w\}| + \sum_{w \ni e} x_e = |diagonals(w)|.$$ 
If $w$ is not incident to $\cc$, then $\{e \in \Dg: e \ni w\}$ is  $diagonals(w)$ by definition, so canceling in the above equation, we get $\sum_{w \ni e} x_e = 0$. If $w$ is incident to $\cc$, then $|\{e \in \Dg: e \ni w\}| = 0$, while $|diagonals(w)| = 1$, so we get $\sum_{e \in E_w} x_e = 1$. Thus our new right-hand side is 0 if $w$ is not incident to $\cc$ and 1 if $w$ is incident to $\cc$.

For statement (iii), refer to Lemma \ref{SaFeatVsGaFeat} and Proposition \ref{facetsPGaFromGa}. Specifically, boundary segments of $T'$ that are not incident to $\cc$ become interior edges of $\GTg$ when the snake graph is formed, and boundary segments are not affected by the shift, so Proposition \ref{facetsPGaFromGa}(i) proves (iii). 

To prove the first case of (iv), suppose in $Q_{\tau_a}$ we have a pair of opposite sides that are boundary segments of $T'$, and both of the other two sides $\{e,f\}$ are in the set $\{\tau_{i_2},\ldots,\tau_{i_{d-1}}\}$. Then by Lemma \ref{SaFeatVsGaFeat}, their labels are not unique in $\GTg$, and by the construction of the snake graph, $e$ and $f$ form opposite exterior edges of a tile in $\GTg$ that has $a$ as the deleted diagonal, $b$ and $c$ as the other two sides, and flanking edges $a$ and $a$ (see figure in Proposition \ref{facetsPGaFromGa}). This places us exactly in the situation of statement (iii) of Proposition \ref{facetsPGaFromGa}. Of sides $a,b$, and $c$, only $a$ corresponds to a crossed diagonal, so replacing $x_a$ with $x_a+1$ in Proposition \ref{facetsPGaFromGa}(iii) proves the desired inequality that forms the first case.

To prove the second and third cases of (iv), suppose in $Q_{\tau_a}$ we have a pair of opposite sides that are boundary segments of $T'$, and at least one of the other two sides $\{e,f\}$ (without loss of generality, say $e$) is not a diagonal in the set $\{\tau_{i_2},\ldots,\tau_{i_{d-1}}\}$. By Lemma \ref{SaFeatVsGaFeat}, $e$ is a unique label in $\GTg$. Again, by construction of the snake graph, $e$ and $f$ form opposite exterior edges of a tile in $\GTg$ that has $a$ as the deleted diagonal, $b$ and $c$ as the other two sides, and flanking edges $a$ and $a$ (see figure in Proposition \ref{facetsPGaFromGa}). This places us exactly in the situation of statement (ii) of Proposition \ref{facetsPGaFromGa}. If $e$ is a boundary segment, then it is not affected by the shift, so Proposition \ref{facetsPGaFromGa}(ii) proves the second case inequality. If $e$ is diagonal $\tau_{i_1}$ or $\tau_{i_d}$, then it is affected by the shift of 1 unit downward, so replacing $x_e$ with $x_e+1$ in Proposition \ref{facetsPGaFromGa}(ii) gives the third case inequality.

\end{proof}

\begin{example}{\rm
Using the same example as in section \ref{Intro}, we will illustrate Theorem \ref{NTgFromTg}. 
\begin{itemize}
\item (i): Since edges 14 and 15 of $T$ do not appear in $T'$, we get $x_{14} = 0$ and $x_{15} = 0$. 
\item (ii): The other equations defining the affine hull of $\NTg$ come from each vertex of $T'$ and whether they are incident to $\cc$. For example, the edges incident to vertex $A$ are $\{1,2,3,4,7\}$, and $A$ is not incident to $\cc$, so we get the equation $x_1+x_2+x_3+x_4+x_7 = 0.$ 
\item (iii): Edges 7, 10, 11, and 12 are boundary segments of the triangulated polygon $T'$ that are not incident to $\cc$, so we get $x_7 \geq 0, x_{10} \geq 0, x_{11} \geq 0, x_{12} \geq 0$.
\item (iv): The pairs of boundary segments of $T'$ that are opposite sides of $Q_{\tau_a}$, where $a$ is a balanced diagonal in $\Dg$, are $\{1,12\}, \{7,11\}, \{7,10\}$, and $\{8,10\}$. The other pairs of opposite sides of each quadrilateral are, respectively, $\{3,13\}, \\
\{3,5\}, \{4,6\}$, and $\{5,9\}$. Since $\{\tau_{i_1},\tau_{i_2},\ldots,\tau_{i_{d-1}}, \tau_{i_d}\} = \{2,3,4,5,6\},$ these pairs fall into the following cases:
\begin{itemize}
\item Both of $\{3,5\}$ are in $\{\tau_{i_2},\ldots,\tau_{i_{d-1}}\}$. They are sides of $Q_{\tau_4}$, whose other two sides are 7 and 11. This gives $x_4+x_7+x_{11} \leq 1$.
\item The edge 13 is a boundary segment of $T'$, so $\{3,13\}$ gives $x_{13} \geq 0$. Similarly, $\{5,9\}$ gives $x_9 \geq 0$. 
\item The edges $\{4,6\}$ are not both in $\{\tau_{i_2},\ldots,\tau_{i_{d-1}}\}$, nor is either one a boundary segment of $T'$, so we get $x_6 \geq -1$.
\end{itemize}
\end{itemize}
Putting this all together, the affine hull and facets of $\NTg$ are given by
\begin{align*}
\text{affine hull:} \quad & x_{14} = 0, \; \; x_{15} = 0 \\
A: \quad & x_1+x_2+x_3+x_4+x_7 = 0 & B: \quad & x_5+x_6+x_7+x_8 = 0 \\
C: \quad & x_8+x_9 = 1 & D: \quad & x_6+x_9+x_{10} = 0 \\
E: \quad & x_4+x_5+x_{10}+x_{11} = 0 & F: \quad & x_3+x_{11}+x_{12} = 0 \\
G: \quad & x_2+x_{12}+x_{13} = 0 & H: \quad & x_1+x_{13} = 1 \\
\text {facets:} \quad & x_7 \geq 0, \; x_{10} \geq 0, \; x_{11} \geq 0, x_{12} \geq 0, \\
& x_{13} \geq 0, \; x_9 \geq 0, \; x_6 \geq -1, \; x_4+x_7+x_{11} \leq 1
\end{align*}
Note that Corollary \ref{numFacets} confirms that we have found all the facets (as computed above, the number of facets is $2d -1-t = 2(5) - 1 - 1 = 8.$)}
\end{example}

\section{Other Remarks and Conjectures} \label{Other}

Empirical evidence suggests that the polytope $\NTg$ contains no lattice points in its relative interior. The author hopes to prove this in a future paper.

If all the frozen variables $\{x_{n+1}, . . . ,x_{2n+3}\}$ (i.e. boundary segments of the polygon) are set equal to 1, the Newton polytope $\NTg$ is less elegant - there is a collapsing of monomials in the cluster expansion, and not every monomial corresponds to a vertex. For example, let $T$ be the triangulated hexagon below. 

\begin{figure}[h!]
\centering
\includegraphics[scale=0.38]{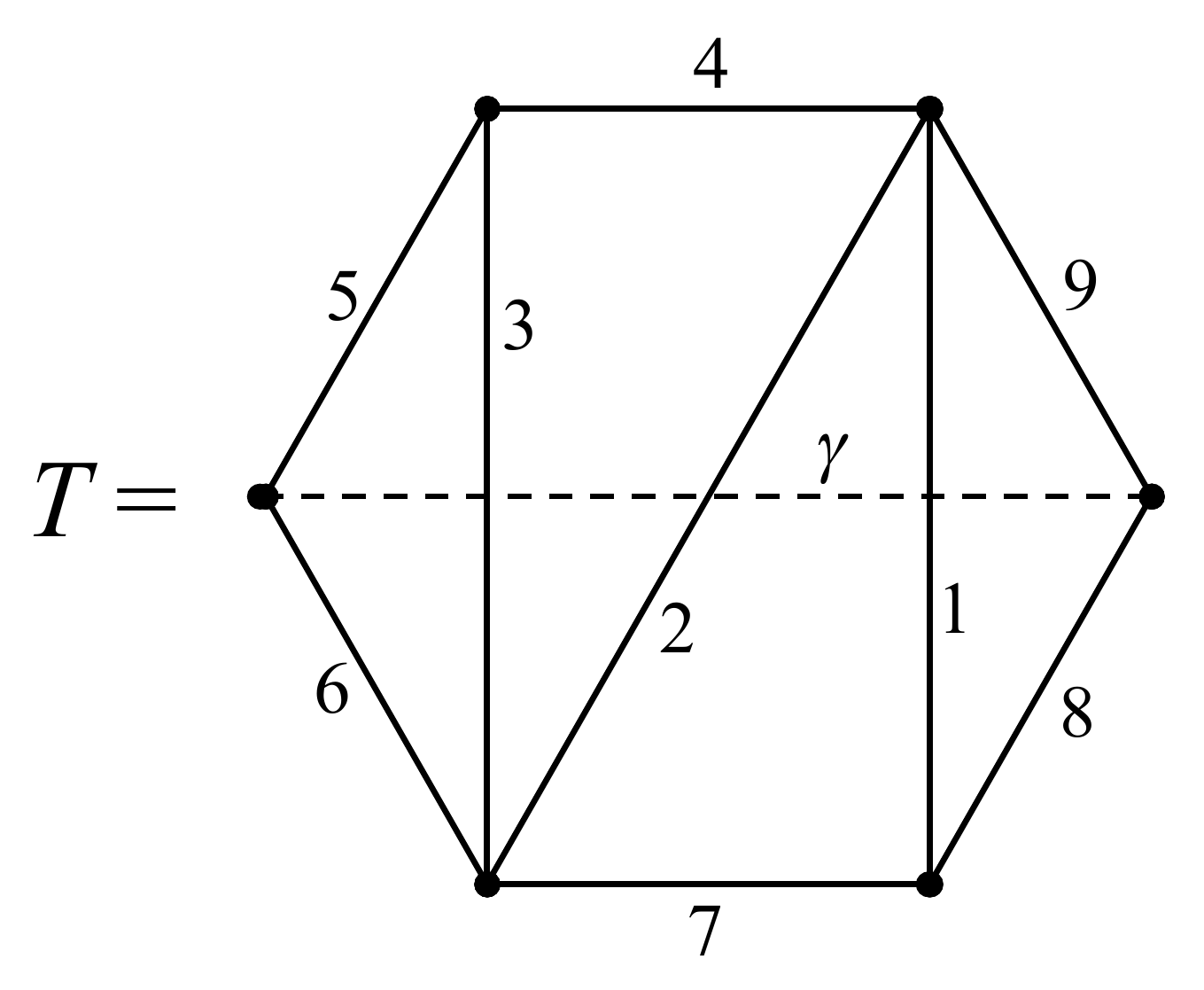}
\label{fig:hexagon}
\end{figure}

The corresponding snake graph is a 2-by-4 square grid graph. The $T$-expansion of $x_\cc$ is

$$\frac{x_2^2 x_5 x_8 + x_2 x_4 x_6 x_8 + x_1 x_3 x_6 x_9 + x_2 x_5 x_7 x_9 + x_4 x_6 x_7 x_9}{x_1 x_2 x_3}$$\\
and $\NTg$ is 3-dimensional with 5 vertices. These numbers correspond neatly to the 3 crossed diagonals and 5 terms in the numerator, as well as the 3 boxes and 5 perfect matchings of the corresponding snake graph. 

However, when the frozen variables $x_4$ through $x_9$ are all set equal to 1 in the above expression, it collapses to

$$\frac{x_2^2 + 2x_2 + x_1 x_3 + 1}{x_1 x_2 x_3}$$\\
and the Newton polytope becomes the convex hull of the four points $(-1,1,-1)$, $(-1,0,-1)$, $(0,-1,0)$, and $(-1,-1,-1)$. So there are 4 rather than 5 points to begin with, and one of them is not even a vertex (the point $(-1,0,-1)$ is the midpoint of the edge joining $(-1,1,-1)$ and $(-1,-1,-1)$). The polytope is also 2-dimensional rather than 3-dimensional (the original square pyramid has collapsed into a triangle). 

The construction of cluster algebras from triangulations of a polygon may be generalized to construct cluster algebras from triangulations of an arbitrary surface with marked points (\cite{FG1, FG2, FG3, GSV, FST}). In this setting, there is a generalization of the Laurent expansion formula using perfect matchings of snake graphs (\cite{MS, MSWp, MSW}). In a future paper, the author hopes to extend the results in this paper to more general surfaces. 

Empirically, many of the results of this paper do not hold when a more general surface is considered. For example, consider the annulus and corresponding snake graph that serves as the example in \cite{MS}, Section 7:

\begin{figure}[h!]
        \centering
        \begin{subfigure}[b]{0.5\textwidth}
                \centering
                \includegraphics[width=\textwidth]{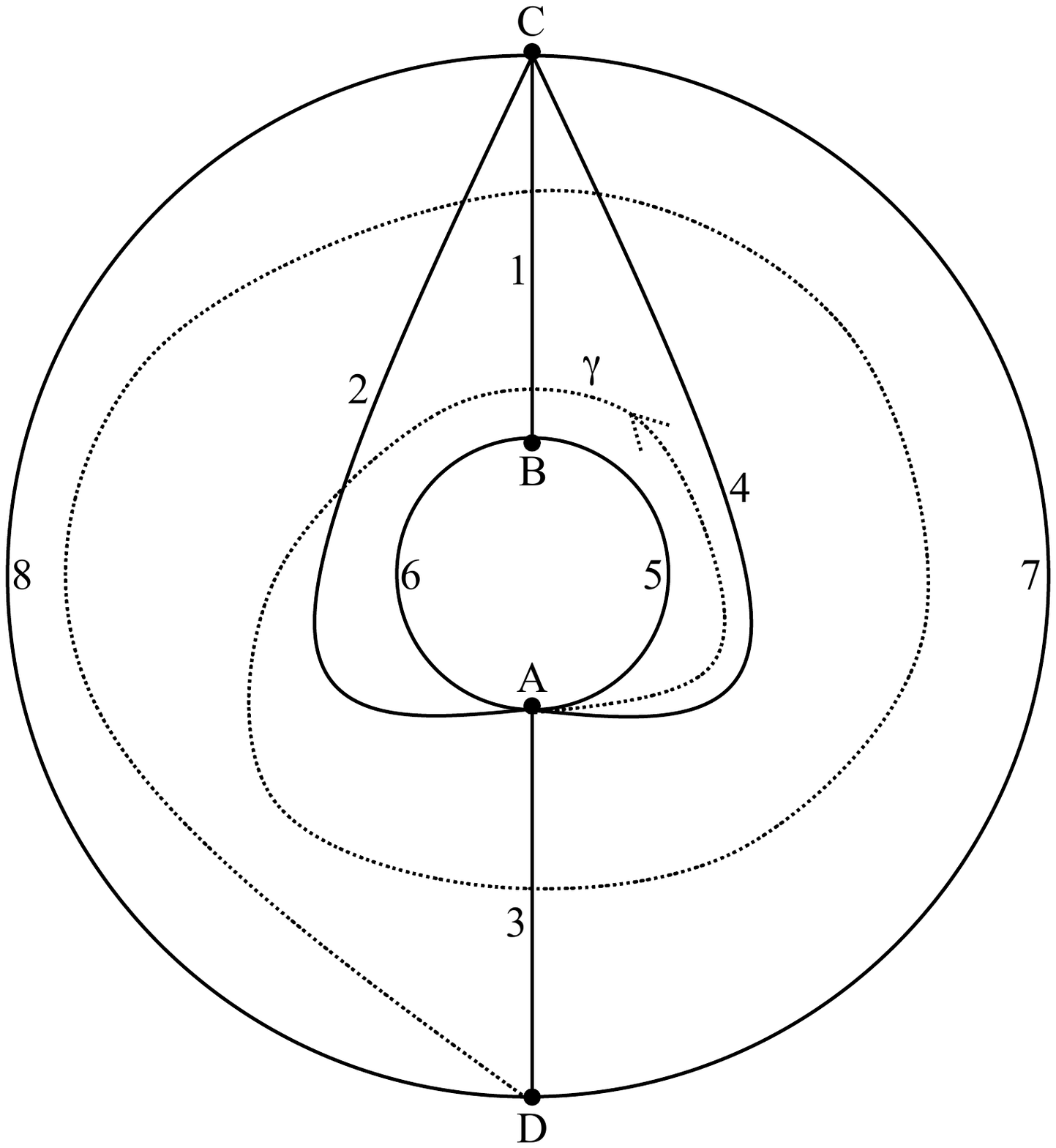}
                \label{fig:annulus}
        \end{subfigure}
        \qquad
        \begin{subfigure}[b]{0.4\textwidth}
                \centering
                \includegraphics[width=\textwidth]{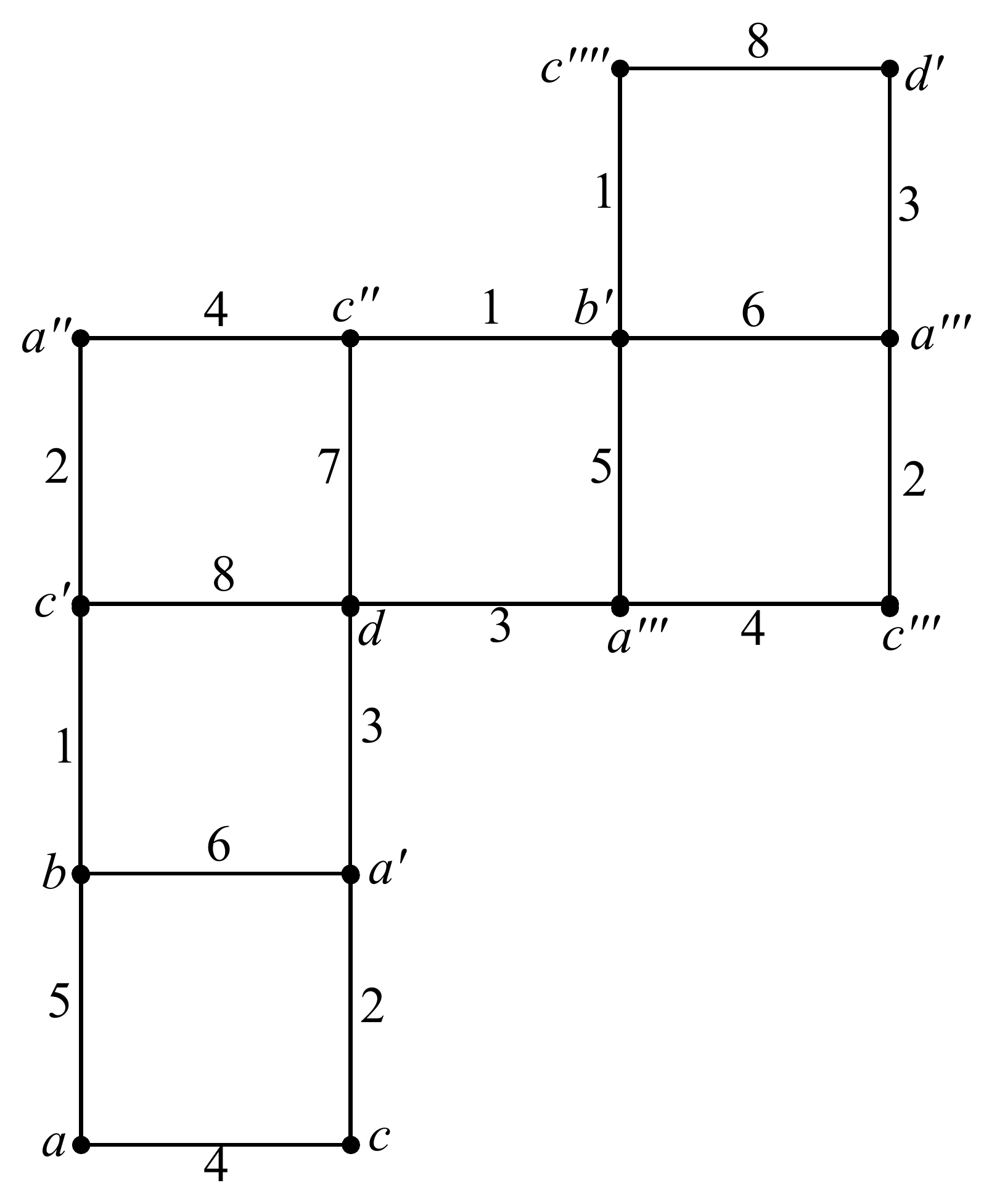}
                \label{fig:MS_G_gamma}
        \end{subfigure}
        \caption{Example from \cite{MS}, Section 7}
        \label{fig:MSexample}
\end{figure}

There are 17 matchings of the snake graph in Figure \ref{fig:MSexample}, but these 17 matchings give only 13 distinct monomials, because two different matchings may give the same monomial. For example, there are two different matchings that both give the monomial $x_1 x_2 x_3 x_4^2 x_5 x_8$. Therefore, there can be no bijection between matchings and vertices of $\GTg$. Moreover, there is no bijection between Laurent monomials and vertices either, because when these 13 distinct monomials are used to form the Newton polytope, only 9 of them correspond to vertices. So overall, our result on the isomorphism of lattices does not hold for more general surfaces. Using the principal coefficient system instead for this example, the same collapsing occurs, even if we leave all $x-$variables and $y-$variables as they are, setting nothing equal to 1. 

For surfaces other than a polygon, our results concerning the facets are not valid, and our affine hull description in Theorem \ref{NTgFromTg}(i)-(ii) seems only partially complete, in that the theorem seems to give some of the affine hull equations, but not necessarily all of them.
Soon, we hope to have a complete affine hull and facet description for cluster variables from more general surfaces.

\end{document}